\documentclass[a4paper,english]{article}

\usepackage[utf8]{inputenc}
\usepackage[T1]{fontenc}
\usepackage{babel}
\usepackage[babel]{csquotes}

\usepackage{amsmath}
\usepackage{amsfonts}
\usepackage{amsthm}
\usepackage{bbm}
\usepackage{amssymb}
\usepackage[top=3cm, bottom=3cm, left=3.3 cm, right=3.3 cm]{geometry}
\usepackage{mathrsfs}
\usepackage{stmaryrd}
\usepackage{lipsum}
\usepackage{hyperref}

\usepackage{float}
\usepackage{tikz}
\usetikzlibrary{shapes.misc}
\tikzset{cross/.style={cross out, draw, 
minimum size=2*(#1-\pgflinewidth), 
inner sep=0pt, outer sep=0pt}}
\usepackage{IEEEtrantools}
\usepackage{url}

\usepackage{enumitem}

\newtheorem{theorem}{Theorem}[section]
\newtheorem{thmx}{Theorem}

\newtheorem{proposition}[theorem]{Proposition}
\newtheorem{lemma}[theorem]{Lemma}
\newtheorem*{namedthm}{\namedthmname}

\newtheorem*{Lemma-drift-union}{Lemma \ref{drift-union} b)}

\theoremstyle{definition}
\newtheorem{definition}[theorem]{Definition}

\newtheorem*{thA'}{Theorem $\bf A'$}
\newtheorem*{thB'}{Theorem $\bf B'$}
\newtheorem{example}{Example}
\newtheorem*{remark}{Remark}

\mathcode`l="8000
\begingroup
\makeatletter
\lccode`\~=`\l
\DeclareMathSymbol{\lsb@l}{\mathalpha}{letters}{`l}
\lowercase{\gdef~{\ifnum\the\mathgroup=\m@ne \ell \else \lsb@l \fi}}
\endgroup

\DeclareMathOperator{\GL}{\mathrm{GL}}

\DeclareMathOperator{\SL}{\mathrm{SL}}

\DeclareMathOperator{\Aut}{Aut}
\DeclareMathOperator{\Stab}{Stab}
\DeclareMathOperator{\Ad}{Ad}
\DeclareMathOperator{\card}{card}

\newcommand{\bracket}[1]{\langle#1\rangle}
\newcommand{\dd}{\mathrm{d}}

\newcommand{\R}{\mathbb{R}}

\newcommand{\Z}{\mathbb{Z}}
\newcommand{\N}{\mathbb{N}}
\newcommand{\Q}{\mathbb{Q}}

\newcommand{\PP}{\mathbb{P}}

\newcommand{\E}{\mathbb{E}}
\newcommand{\1}{\mathbbm{1}}

\newcommand{\cS}{\mathcal{S}}

\newcommand{\cA}{\mathcal{A}}
\newcommand{\cF}{\mathcal{F}}

\newcommand{\cT}{\mathcal{T}}

\newcommand{\ka}{\mathfrak{a}}
\newcommand{\kl}{\mathfrak{l}}
\newcommand{\kt}{\mathfrak{t}}

\newcommand{\ks}{\mathfrak{s}}
\newcommand{\kr}{\mathfrak{r}}
\newcommand{\g}{\mathfrak{g}}
\newcommand{\h}{\mathfrak{h}}

\newcommand{\eps}{\varepsilon}

\newcommand{\leb}{\text{leb}}

\newcommand{\abs}[1]{\lvert#1\rvert}    
\newcommand{\Abs}[1]{\left\lvert#1\right\rvert}    
\newcommand{\norm}[1]{\lVert#1\rVert}   

\newcommand{\cSS}{\cS_{\Omega}(\Gamma)}

\newcounter{namedthm}

\makeatletter
\newenvironment{named}[1]
  {\def\namedthmname{#1}%
   \refstepcounter{namedthm}%
   \namedthm\def\@currentlabel{#1}}
  {\endnamedthm}
\makeatother

\title{Random walks with bounded first moment on finite-volume spaces}
\author{Timothée Bénard\thanks{The first author has received funding from the European Research
Council (ERC) under the European Union’s Horizon 2020 research and
innovation programme (grant agreement No. 803711).}
\hspace{1pt}
and Nicolas de Saxcé
}
\date{}
\begin{document}

\maketitle

\large
\begin{abstract} 
Let $G$ be a real Lie group, $\Lambda\leq G$ a lattice,  and $\Omega=G/\Lambda$. We study the equidistribution properties of the left random walk on $\Omega$ induced by a probability measure $\mu$ on $G$. It is assumed that $\mu$ has a finite first moment,  and that the Zariski closure of the group generated by the support of $\mu$ in the adjoint representation is semisimple without compact factors.  We show that for every starting point $x\in \Omega$, the $\mu$-walk with origin $x$ has no escape of mass, and equidistributes in Cesàro averages toward some homogeneous measure. This  extends several fundamental results due to Benoist-Quint and Eskin-Margulis for walks with finite exponential moment. 
\end{abstract}

\tableofcontents

\bigskip

\bigskip

\pagebreak

\section{Introduction}

A homogeneous Markov chain on a finite-volume real homogeneous space is formally incarnated by  a triple  $(G, \Lambda, \mu)$ where $G$ is a real Lie group, $\Lambda$ a discrete subgroup of finite covolume  in $G$, and $\mu$ a Borel probability measure on $G$. The chain in question then corresponds to the left $\mu$-random walk on the quotient $\Omega=G/\Lambda$. In other words, the transition law at a point $x\in \Omega$ is given by the convolution $\mu\ast \delta_{x}$, image of $\mu$ under $g\mapsto gx$. 

In the last 20 years, the study of walks on finite-volume spaces has known spectacular advances, which were achieved in analogy with Ratner's theorems describing the dynamics of Ad-unipotent flows on $\Omega$.
The first milestone was set by Eskin-Margulis who proved that the $n$-th step distribution of a homogeneous chain essentially remains in a compact set \cite{EskMar04}. A few years later, Benoist-Quint  managed to classify all the $\mu$-invariant probability measures \cite{bq1, bq2}, and extrapolated  in  \cite{bq3}  that almost every trajectory of the random walk equidistributes in some finite-volume  homogeneous subspace.
However,  all these statements require stringent moment assumptions on the measure $\mu$, such as \emph{compact support} or \emph{finite exponential moment}, the latter meaning that for some $\alpha>0$, we have
\[
\int_{G} \norm{\Ad g}^\alpha \, \dd\mu(g) <\infty.
\]
In a recent paper, Eskin and Lindenstrauss \cite{el_rw} extended the techniques developed by Benoist and Quint, and showed that their measure classification was still valid in the case where $\mu$ only has a \emph{finite first moment}.
Our goal is to weaken in the same way the moment assumptions in the Eskin-Margulis recurrence theorem and in the Benoist-Quint equidistribution theorems.
This answers a question formulated by Benoist-Quint in the $10^{\text{th}}$ Takagi Lectures \cite[Question~2]{bq_intro}.

\subsection{Main results}

Let $G$ be a real Lie group, $\Lambda$ a lattice in $G$, and set $\Omega=G/\Lambda$. We fix a Borel probability measure $\mu$  on $G$ and denote by $\Gamma$ the semigroup generated by the support of $\mu$. The algebraic group  generated by its adjoint representation is denoted by $H=\overline{\Ad \Gamma}^Z$, and we call $H^{nc}$ its non-compact part, defined as the smallest normal algebraic subgroup  of $H$ such that $H/H^{nc}$ is compact.
All the theorems to follow  are presented under the next two assumptions.

\begin{enumerate}

\item The non-compact part $H^{nc}$ of $H$ is semisimple. 

\item The measure $\mu$ has a \emph{finite first moment}: $$\int_{G}\log \norm{\Ad g} \, \dd\mu(g) <\infty$$

\end{enumerate}

To state our results, we need to introduce the notion of homogeneous subspace. 

\begin{definition}
 A closed subset $Y$ of $\Omega$ is \emph{homogeneous} if its stabilizer $G_{Y}=\{g\in G, \,gY=Y\}$ acts transitively on $Y$. We add that $Y$ has \emph{finite volume} if the action of $G_{Y}$ on $Y$ preserves a Borel probability measure on $Y$. Such a measure is then unique and denoted by $\nu_{Y}$. If the semigroup $\Gamma$ is included in $G_{Y}$ (and acts ergodically on $(Y,\nu_{Y})$), we say that  $Y$ is \emph{$\Gamma$-invariant} (and $\Gamma$-\emph{ergodic}). 

\end{definition}

We denote by $\cS_\Omega(\Gamma)$  the set of $\Gamma$-invariant ergodic finite-volume closed homogeneous subsets of $\Omega$.

\bigskip

Our first result states that the $\mu$-walk on $\Omega$ essentially evolves in a compact subset.
Such statements originate in the work of Eskin-Margulis-Mozes~\cite{EskMarMoz98} on the quantitative Oppenheim conjecture and can be seen as an analog for random walks of the Dani-Margulis recurrence theorem on Ad-unipotent flows \cite{Dan}.

\begin{thmx}[Non-escape at infinity]
\label{thA}
For every compact set $K\subset\Omega$ and every $\eps>0$, there exists a compact set $K'\subset \Omega$ such that for every $x\in K$,
\begin{enumerate}[label=(\roman*)]
\item for every $n\geq 0$, $(\mu^{*n}*\delta_x)(K') >1-\eps$;
\item for $\mu^{\otimes\N^*}$-almost every instructions $(g_i)_{i\geq 1}$,
\[
\liminf_{n\to+\infty}\frac{1}{n}\card\{k\in\{1,\dots,n\}\ |\ g_k\dots g_1x \in K'\} >1- \eps.
\]
\end{enumerate}
\end{thmx}
 Conclusion (i) means the mass of the $n$-th step distribution of the walk does not escape at infinity. Conclusion (ii) expresses some positive recurrence of the walk's trajectories. 
The result actually holds under the slightly weaker assumption that the image measure $(\Ad_{\ks})_{*}\mu$ of $\mu$ under the adjoint representation on the largest semisimple  quotient  of $\g$ without compact factors generates an algebraic group with semisimple non-compact part (see \ref{thA'} in Section \ref{Sec4}). This theorem generalizes \cite{bq_fv} and \cite{EskMar04}, which assume that $\mu$ has a finite exponential moment. 

\bigskip

Our second result states that the $\mu$-walk  on $\Omega$ does not accumulate on  a $\mu$-invariant homogeneous subspace, unless it is trapped inside it.

\begin{thmx}[Unstability of  invariant homogeneous subspaces]
  \label{thB}

Let $Y\in\cS_\Omega(\Gamma)$ and consider  a compact subset $K\subseteq \Omega\smallsetminus Y$.
For every $\eps>0$, there exists a neighborhood $O'$ of $Y$ in $\Omega$  such that for all $x\in K$, 
\begin{enumerate}[label=(\roman*)]
\item \label{fer} for every $n\geq 0$, $(\mu^{*n}*\delta_x)(O')<\eps$ ;
\item\label{ces} for $\mu^{\otimes\N^*}$-almost every instructions $(g_i)_{i\geq 1}$,
\[
\limsup_{n\to+\infty}\frac{1}{n}\card\{k\in\{1,\dots,n\}\ |\ g_k\dots g_1x\in O'\}<\eps.
\]
\end{enumerate}
\end{thmx}
For the purpose of studying the equidistribution of $\mu$-trajectories,  we prove a more general version (\ref{thB'}) where $Y$ is replaced by its orbit $L_{0}Y$ under a compact subset $L_{0}$ of the centralizer of $H^{nc}$ in $G$.
Conclusion \ref{ces} is obtained in \cite{bq3} under the assumption that $\mu$ has a finite exponential moment. If the space $\Omega=G/\Lambda$ is compact, item~\ref{fer} can be readily deduced from the arguments given in \cite[Section~6]{bq2}, but it appears to be new when $\Omega$ is not compact, even under exponential moments assumptions. 

\bigskip
Our third result states that  each $\mu$-walk trajectory on $\Omega$ equidistributes in its closure toward some homogeneous measure. It can be seen as an analog for random walks  of Ratner's equidistribution theorem  \cite{Rat} for Ad-unipotent flows on $\Omega$. As observed in \cite{bq2, el_rw}, the assumption on $H$ must be strengthened in order to guarantee homogeneity.
We ask that $H$ be Zariski connected semisimple without compact factors, or equivalently $H=H^{nc}$ semisimple.

\begin{thmx}[Equidistribution in Cesàro-averages]
\label{thC}
Suppose as well $H=H^{nc}$. For every $x\in\Omega$, we have:
\begin{enumerate}[label=(\roman*)]
\item \label{oc} The orbit closure $Y=\overline{\Gamma\cdot x}\subset\Omega$ is a $\Gamma$-invariant ergodic finite-volume closed homogeneous subset.
\item \label{ce} The sequence of measures $(\frac{1}{n}\sum_{k=0}^{n-1}\mu^{*k}*\delta_x)_{n\geq 1}$ converges to $\nu_Y$ in the weak-$\ast$  topology.
\item \label{tr} For $\mu^{\otimes \N^*}$-almost every instructions $(g_i)_{i\geq 1}$, the sequence of empirical measures $(\frac{1}{n}\sum_{k=0}^{n-1}\delta_{g_k\dots g_1 x})_{n\geq 1}$ converges to $\nu_Y$ in the weak-$\ast$  topology.
\end{enumerate}
\end{thmx}

Theorem \ref{thC} follows from \cite{bq3} in the case where $\mu$ has a compact support, which implies in particular that $\Gamma$ is compactly generated. It is a well-known conjecture that the Cesàro-averages in the second item should not be necessary. This  will be proven in a follow-up paper \cite{Benard21} in the case where two powers of $\mu$ are not mutually singular.

\subsection{Dynamics of the proofs}

In order to prove Theorem~\ref{thA}, we use a variant of Foster's recurrence criterion for walks with a negative drift, applied to an appropriate proper drift function on the space of lattices.
This strategy is generally credited to Margulis, and goes back to the works~\cite{EskMarMoz98} and \cite{EskMar04}.
It was further developed by Benoist and Quint~\cite{bq_fv}.
In those papers, the authors make an exponential moment assumption, and use the elementary Proposition~\ref{rec-EM} below\footnote{To be more accurate, those papers rather use a contraction property of the form $\E_x[F(X_1)]\leq aF(x) + b$, for some $a\in(0,1)$ and $b\in\R^{+}$.
It is not hard to see that under the assumptions of the above proposition, the function $F(x)=e^{\delta f(x)}$ where $\delta>0$ is very small, has the desired contraction property.
The recurrence criterion under the contraction property is sometimes called \enquote{geometric recurrence criterion}, or \enquote{exponential recurrence criterion}.
}, whose first item is due to Eskin-Margulis  \cite[Lemma 3.1]{EskMar04} and the second to Benoist-Quint \cite[Proposition 3.9]{bq3}.

\begin{proposition}[Eskin-Margulis, Benoist-Quint]
\label{rec-EM}
Let $\Omega$ be a measurable space, $f: \Omega\rightarrow \R^+$ a measurable function and $(X_{n})_{n\geq 0}$ a measurable Markov chain on $\Omega$. Assume 
\begin{itemize}
\item[$(\mathbf{EM})$] $\exists \alpha_{0}>0$, 
$\,\sup_{x\in \Omega}\mathbb{E}_{x}\left(e^{\alpha_{0}|f(X_{1})-f(x)|}\right)<\infty$
\item[$(\mathbf{D})$]  $\exists R_{0}, \lambda_{1}>0,$ $\forall x\in f^{-1}(R_{0}, +\infty)$,
$ \,\mathbb{E}_{x}(f(X_{1}))\leq f(x)-\lambda_{1}  $
\end{itemize}
Then for all $\eps>0$, there exists $R>0$ such that for  $x \in \Omega$, 
\begin{enumerate}[label=(\roman*)]
\item\label{pro}
\(
 \forall n\geq n_{x},\quad   \PP_x[f(X_n)>R] \,<\, \eps 
\)
\item\label{tra}
\(
 \PP_x-\mbox{almost surely},
\)
\[
\limsup_{n\to+\infty}\frac{1}{n}\card\{k\in\{1,\dots,n\}\ |\ f(X_k)>R\} < \eps.
\]
\end{enumerate}
\end{proposition}

In order to remove the exponential moment condition in the Benoist-Quint theory, our most important task will be to obtain an analogous recurrence criterion that applies to homogeneous random walks with only a finite first moment assumption.

A first approach would be to simply try weakening the requirement ($\mathbf{EM}$) above.
And indeed we shall see that Proposition \ref{rec-EM} is still true if we replace  ($\mathbf{EM}$) by the condition that the increments of the walk are uniformly bounded in $L^{1+\eta}$ for some $\eta>0$:
\begin{itemize}
\item [$(\mathbf{M}_{\eta})$]  $\exists M>0$, $ \forall x\in \Omega$, 
$\,\mathbb{E}_{x}\left(|f(X_{1})-f(x)|^{1+\eta}\right)<M.$
\end{itemize}

However, carefully constructed (continuous) Markov chains on $\R^+$ show that \emph{conclusions (i) and (ii) may not hold for walks with negative drift and increments uniformly bounded in $L^1$}, in other words if ($\mathbf{EM}$) above is replaced by ($\mathbf{M}_{0}$) (see Section \ref{Sec2}, Examples \ref{counterex-L1} and  \ref{counterex2-L1}). We need a new point of view, which  reflects  more refined properties of homogeneous chains than the drift in expectation expressed by $(\mathbf{D})$. Our crucial claim is that the height of a homogeneous walk evolving  in the cusps of a finite-volume space $\Omega=G/\Lambda$ must decrease faster than a chain on $\R$ with i.i.d. increments of negative mean. We formalize this concept by the notion of stochastic dominance. 

\begin{definition}
Given real random variables $Z$ and $(Z_{i})_{i\in I}$, we say that \emph{$(Z_{i})_{i\in I}$ is stochastically dominated by $Z$} if  for every $i\in I$, $t\in \R$\[
 \PP(Z_i >  t) \,\leq \,\PP(Z>  t).
\]
\end{definition}
The abstract version  of Proposition \ref{rec-EM} we shall need is the following. 
\begin{thmx}[Recurrence of Markov chains with stochastic dominance]
\label{thD} 
Let $\Omega$ be a measurable space, $f: \Omega\rightarrow \R^+$ a measurable function, and $(X_{n})_{n\geq 0}$ a measurable Markov chain on $\Omega$. 

Assume there exist a sublevel set $K:=\{f\leq R_{0}\}$ and integrable real random variables $Z_0$, $Z_1$ with $\E[Z_1]:=-\lambda_1<0$   such that 
\begin{enumerate}
\item $(f(X_1)-f(X_0)|X_0=x)_{x\in K}$ is stochastically dominated by $Z_0$;
\item $(f(X_1)-f(X_0)|X_0=x)_{x\in \Omega \smallsetminus K}$ is stochastically dominated by $Z_1$.
\end{enumerate}
Then for all $\eps>0$, there exists $R>0$ such that for all $x \in \Omega$, 
\begin{enumerate}[label=(\roman*)]
\item\label{pro}
\(
 \forall n\geq 1,\quad   \PP_x[f(X_n)>R] \,<\, \eps + \frac{f(x)}{n\lambda_{1}}.
\)
\item\label{tra}
\(
 \PP_x-\mbox{almost surely},
\)
\[
\limsup_{n\to+\infty}\frac{1}{n}\card\{k\in\{1,\dots,n\}\ |\ f(X_k)>R\} < \eps.
\]
\end{enumerate}
\end{thmx}

It is pleasant to note that conclusion (i) gives an \emph{optimal bound} on the time $n$ at which the distribution $X_{n}$ starts to accumulate on  the sublevel set $\{f\leq R\}$ (see discussion below Proposition \ref{rec-law}).

\bigskip

Once Theorem \ref{thD} is established, we use it to show that homogeneous random walks do not escape at infinity, namely Theorem \ref{thA}.
Theorem \ref{thD} will be applied to a proper drift function $f:\Omega\to\R^+$ extracted from Benoist-Quint's paper \cite{bq_fv}. Benoist-Quint's function is itself inspired by a former construction due to Eskin, Margulis, and Mozes \cite{EskMar04, EskMarMoz98}.  Checking the stochastic condition of Theorem \ref{thD}  will require some work, mostly done in Section \ref{Sec3}. 

\bigskip

The proof of Theorem \ref{thB} is inspired by \cite[Section~6]{bq2} which obtains  item (ii) when $\mu$ has a finite exponential moment.   However, several important changes are needed. The approach of \cite{bq2} consists in showing that the first return random walk  induced on a large compact subset $Q$ drifts away from $Y\cap Q$  in expectation. It is not possible to extract information on the original $n$-th step distribution $\mu^n \ast \delta_{x}$ from this induced random walk, so conclusion (i) cannot be obtained with such a strategy. Moreover, checking that the induced walk does satisfy the new conditions of stochastic dominance formulated in Theorem \ref{thD} would raise other significant difficulties.

Our solution instead is to construct a global drift function for every  closed invariant finite-volume homogeneous subset $Y$; this is done by gluing together the Benoist-Quint drift function with a function that drifts away from $Y$ on a compact subset.
We believe this new technique could be useful to construct drift functions in other contexts.

\bigskip
Finally, the equidistribution of each $\mu$-trajectory in its closure stated in Theorem \ref{thC} results from a combination of  Theorems \ref{thA}, \ref{thB},   the Eskin-Lindenstrauss classification of stationary probability measures  \cite[Theorem~1.3]{el_rw},  and the general strategy of \cite{bq3}. The proof roughly goes as follows. Consider a starting point $x\in \Omega$, a typical sequence of instructions $(g_{i})_{i\geq 1}$ and write for $n\geq 1$,
\[
\nu_{n}=\frac{1}{n}\sum_{k=0}^{n-1}\delta_{g_k\dots g_1 x}.
\]
 Breiman's law of large numbers and  Theorem \ref{thA} imply that any weak-$\ast$ limit $\nu$ of  $(\nu_{n})_{n\geq 1}$ is a $\mu$-stationary probability measure. By Theorem \ref{thB}, $\nu$ does not give mass to the $\Gamma$-invariant homogeneous subspaces which do not contain $x$.  
The Eskin-Lindenstrauss classification  of stationary measures implies that $\nu=\nu_{Y}$ for some $Y\in S_{\Omega}(\Gamma)$ and necessarily  $Y:= \overline{\Gamma.x}$. This yields  the equidistribution statement. 

\bigskip

\noindent\textbf{Structure of the paper}

Section \ref{Sec2} is dedicated to Theorem \ref{thD}.
We first explain in greater detail the role of the dominance hypothesis and then prove the result, subdivided into two propositions (Propositions~\ref{rec-law} and \ref{rec-traj}).
We also check that the conclusions still hold for walks with a negative drift and increments bounded in $L^{1+\eta}$.

Section \ref{Sec3} converts Theorem  \ref{thD} into a more handleable recurrence criterion for homogeneous random walks.
It will be used in Sections \ref{Sec3} and \ref{Sec4} to show that homogeneous walks accumulate neither at infinity, nor around any proper invariant homogeneous subspace.  

Section \ref{Sec4} deals with Theorem \ref{thA}.
We check that the Benoist-Quint drift function satisfies the controlled drift assumption of Section~\ref{Sec3}.

Section \ref{Sec5} deals with Theorem \ref{thB}. The main part of the argument is the construction of a new drift function for invariant homogeneous closed subsets. It is obtained by gluing the Benoist-Quint function with some function which drifts away from the homogeneous subspace on a compact set. 

Section \ref{Sec6} contains the proof of Theorem \ref{thC}.
It combines Theorems~\ref{thA}, \ref{thB} and the classification  of  stationary measures of Eskin-Lindenstrauss.

Section \ref{Sec7}  concludes the paper with some possible further directions of research.

\bigskip

\noindent\textbf{Acknowledgements.}
We are indebted to Yves Benoist and Jean-François Quint for several helpful discussions on the subject.
It is a pleasure to thank them here for sharing their knowledge, and for organizing the \emph{Arbeitsgemeinschaft} on rigidity of stationary measures at the Mathematisches Forschungsinstitut Oberwolfach, where we learned about these problems.
We also thank Jean-Philippe Dru for his help in finding reference~\cite{Tweed83}.

\section{Markov chains with a negative drift}
\label{Sec2}

The goal of this section is to prove Theorem \ref{thD} stating that Markov chains on a stratified state space and satisfying some stochastic dominance condition have \emph{neither  escape of mass  nor escape  of empirical measures}.
We also note that these  conclusions are still valid if the chain is assumed to have a negative drift and increments uniformly bounded in $L^{1+\eta}$ for some $\eta>0$.
Counterexamples are given if $\eta=0$.

Recurrent aspects of general Markov chains is a classical subject in probability theory, and the books \cite{Rev84} and \cite{MeyTweed12} will provide a thorough introduction to the subject to the interested reader.
We note that many recurrence results for Markov chains on continuous state spaces make some irreducibility assumption such as $\psi$-irreducibility, or Harris-recurrence.
Those are natural assumptions that remove the dependency of the Markov chain on the starting point, and they are certainly necessary to obtain some convergence statement such as \cite[Theorem~13.0.1]{MeyTweed12}, but they are not satisfied by random walks on homogeneous spaces.

However, in this section, we shall only be interested in weaker statements such as non-escape of mass, and for that, one can compensate the lack of irreducibility by making a stronger moment assumption, or a stochastic dominance assumption. 
It is only through the  classification of stationary measures \cite{bq2, el_rw}, a result that relies heavily on the specific structure of homogeneous spaces, that we shall be able to derive an equidistribution theorem from the result on non-escape of mass.

\smallskip
For the whole section, we  fix a measurable space $\Omega$,  a non-negative  measurable  function $f:\Omega \rightarrow \R^{+}$, and a Markov chain  $(X_{n})_{n\geq 0}$   on $\Omega$. Given a point $x\in \Omega$, the notations $\mathbb{E}_{x}$ and $\mathbb{P}_{x}$  will refer to the expectation and probability conditional to $X_{0}=x$.

\subsection{Dominance condition and return times} \label{Sec2.1}

We first give some perspective to the stochastic dominance assumptions in Theorem \ref{thD}. For future reference, these assumptions will be denoted by $(\mathbf{SD})$. Recall for clarity:

\begin{itemize}
\item [$(\mathbf{SD})$]
There exist a sublevel set $K:=\{f\leq R_{0}\}$ and integrable real random variables $Z_0$, $Z_1$ with $\E[Z_1]:=-\lambda_1<0$   such that  every $x\in \Omega$, $t\in \R$,
\[
\mathbb{P}_{x}(f(X_{1})-f(x) > t )\,\,\leq\,\, \mathbb{P}(Z_{0}> t)\1_{K}(x) \,+\,\mathbb{P}(Z_{1}> t)\1_{\Omega\smallsetminus K}(x).
\]
\end{itemize}

We begin by observing that  $(\mathbf{SD})$ is a very natural condition as it is satisfied for Markov chains with a well-behaved negative drift. This criterion  is expressed in Lemma  \ref{crit-SD} below.  It will  further be adapted to homogeneous chains in Section \ref{Sec3}. The statement uses the notion of standard realisation of a random variable which we first recall. 

\begin{definition}[Standard realisation]
Let $Z$ be a real random variable.
The random variable on the probability space $([0,1], \mathscr{B}([0,1]), \leb_{|[0,1]})$ defined by 
\[
\begin{array}{lccl}
Z': & [0,1] & \rightarrow & \R\\
& s & \mapsto & \max\{t\in \R, \,\, \mathbb{P}(Z\geq t)\geq s\}
\end{array}
\]
has the same law as $Z$. We call it the \emph{standard realisation} of $Z$.
\end{definition}

\begin{remark}
This notion allows a second formulation of stochastic dominance:  $Z_1$ is dominated by $Z_2$ if and only if their standard realisations $Z_1', Z_2'$ satisfy $Z_1'\leq Z_2'$ everywhere on $[0,1]$.
\end{remark}

\begin{lemma}[Criterion for stochastic dominance]
\label{crit-SD}
Let $(\Omega, f, (X_{n}))$ as above. Assume there exist a probability space $(E, \mathbb{P})$, an integrable random variable  $Z : E \rightarrow \R_{+}$ and constants $R_{0}, \lambda, \alpha>0$
such that 
\begin{enumerate}
\item $\forall x \in \Omega$, 
$$f(x)> R_{0} \implies \mathbb{P}_{x}(f(X_{1}) \leq f(x) -\lambda )\geq 1-\alpha $$
\item $\forall x \in \Omega$, $t\in \R_{+}$, 
$$\mathbb{P}_{x}(f(X_{1}) - f(x) > t )\leq  \mathbb{P}(Z > t ) $$
\item Denoting by $Z'$ the standard realisation of $Z$ we have 
$$\mathbb{E}(Z'\1_{[0,\alpha]}) <\lambda(1-\alpha) $$
\end{enumerate}
Then $(\Omega, f, (X_{n}))$ satisfies the  dominance condition $(\mathbf{SD})$ with constants $(R_{0}, \lambda_{1})$ where $\lambda_{1}:=\lambda(1-\alpha) -\mathbb{E}(Z'\1_{[0,\alpha]})$.
\end{lemma}

\begin{proof}
Given $x\in \Omega$, we denote by $Y_{x}$ the  variable $f(X_{1})-f(x)$ varying under $\mathbb{P}_{x}$.
Assumption \emph{2.} of the lemma yields that for any $x\in \Omega$, $Y_{x}$ is stochastically dominated by $Z$.  In particular, we may choose $Z_{0}=Z$.  

We now define $Z_{1}$.  Let $x\in \Omega\smallsetminus K$ and $Y'_{x},Z': [0,1]\rightarrow \R$ be the standard realisations of $Y_{x},Z$ defined above. By assumption 1), we have  $Y'_{x}\leq -\lambda$ on $(\alpha, 1]$, hence we can write everywhere on $[0,1]$ 
$$Y'_{x}\leq Z'\1_{[0,\alpha]}  -\lambda \1_{(\alpha, 1]}$$
The right-hand side defines a random variable $Z_{1}$ independently of the point $x\in \Omega\smallsetminus K$ chosen earlier. Moreover,  assumption \emph{3.} yields $$\mathbb{E}(Z_{1}) \leq \E(Z'\1_{[0,\alpha]}) -\lambda(1-\alpha) <0 $$ which concludes the proof. 
\end{proof}

\begin{remark}
The idea behind Lemma \ref{crit-SD} can  easily be understood through the graphs of the repartition functions, as in Figure~\ref{fx} below. Recall that the repartition function of a real random variable $X$ is defined by $F_X(t)=\PP(X\leq t)$ for $t\in \R$.
\begin{figure}[H]
\label{fx}
\begin{center}
\begin{tikzpicture}[domain=.65:6]
\draw[->] (-5,0) -- (5.5,0) node[below] {$t$};
\draw[-] (-5,3) node[left] {$1$} -- (5.5,3);
\draw[-, color=gray] (-5,2.5) node[left] {$1-\alpha$} -- (5.4,2.5);
\draw[->] (1,-.1) node[below] {$0$}-- (1,4.2);
\draw[-,color=gray] (-1,-.1) node[below]{$-\lambda$} -- (-1,4);
\draw[color=green] (5,3.3) node[right] {$F_{f(X_1)-f(x)}(t)$}; 
\draw[color=green, thick] (-5, .2) .. controls (-3,.1) and (-2,2.1) .. (-1,2.6) .. controls (-.5,2.85) .. (5.5,2.95); 
\draw[color=blue, thick] (-5, .01) -- (1, .01) .. controls (1.2,1) and (2,2.2) .. (3,2.5) .. controls (4,2.8) .. (5.5,2.86);
\draw[color=red,thick] (-5, .05) -- (-1,.05); 
\draw[color=red, thick] (-1,2.52) -- (3,2.52);
\draw[color=red, thick] (3,2.54) .. controls (4,2.84).. (5.5,2.9) node[right] {$F_{Z_1}(t)$};
\draw[color=blue] (5.4,2.5) node[right] {$F_{Z}(t)$};
\end{tikzpicture}
\end{center}
\caption{Stochastic domination by $Z_1$ with $\E[Z_1]<0$.}
\end{figure}
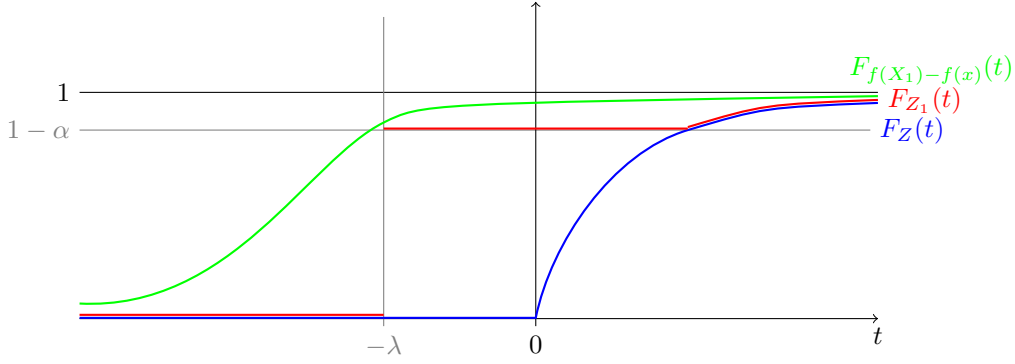
One can readily compute the expectation of $Z_1$ from the above picture:
\begin{align*}
\E[Z_1] & =\int_{\R}t\,\dd F_{Z_1}(t)\\
&  = - (1-\alpha)\lambda + \int \1_{\{F_Z(t)\geq 1-\alpha\}}\dd F_Z(t)\\
& = - (1-\alpha)\lambda + \E[\1_{[0,\alpha]}Z'] = -\lambda_1 <0.
\end{align*}
\end{remark}

 
We now turn to the first implications of condition $(\mathbf{SD})$. Its crucial input is that it allows to \emph{bound the tail probabilities of the return time to the sublevel set $K$} independently of the starting point $x\in K$, and so that the resulting sequence of bounds is summable. This will be the key ingredient to show that the mass of the $n$-th step distribution and the empirical measures do not escape at infinity.
Such bounds are also available if the walk has a negative drift and increments uniformly bounded in $L^{1+\eta}$ for some $\eta>0$, so the non-escape estimates are equally valid in this case.
This latter observation will not be used in the other sections, but seems general enough to be of independent interest.
It also extends the results based on exponential moment conditions used in \cite{EskMar04, bq3}.

\begin{lemma}[Return times]\label{return-times}
Suppose $(\Omega, f, (X_{n}))$  satisfy condition $(\mathbf{SD})$, or conditions $(\mathbf{D})$ and $(\mathbf{M}_{\eta})$ for some $\eta>0$. Denote by
\[
\tau = \inf \{ n\geq 1\ |\ X_{n}\in K\}
\]
the first return time to $K$.
Then 
\[
\sum_{n\geq 1}\sup_{x \in K}\mathbb{P}_{x}(\tau\geq n) <\infty.
\]
\end{lemma}

\bigskip
A weaker version of this result can be directly deduced from the so-called Foster's recurrence criterion, which only assumes negative drift and bounds the expectation of the return time $\tau$. This criterion, or at least an exponential variant, was used  in \cite{bq2, bq3}. It will also play a role in several proofs below so we record the precise statement. 

\begin{lemma}[Foster's criterion]\label{Foster}
Suppose $(\Omega, f, (X_{n}))$  satisfy conditions $(\mathbf{D})$ and $(\mathbf{M}_{0})$. Then for every $x\in \Omega\smallsetminus K$, we have 
$$\E_{x}(\tau)\leq \frac{f(x)}{\lambda_{1}}. $$
\end{lemma}  
In particular, the expectation of $\tau$ is uniformly bounded when the starting point of the chain varies in $K$, or in other terms $\sup_{x \in K}\sum_{n\geq 1}\mathbb{P}_{x}(\tau\geq n) <\infty$. The  goal of Lemma \ref{return-times} is precisely to strengthen this inequality by allowing to  switch the supremum and the sum.
\begin{proof}[Proof of Lemma \ref{Foster}] It is a classical result, see for instance  \cite[Proposition 11.3.2]{MeyTweed12}, or Foster's original paper \cite{Fost53}. We recall the proof for completeness. Let $x\in \Omega\smallsetminus K$ be a starting point. For $n\geq0$, set $Y_n=\1_{\tau>n}f(X_n)$, and $\cF_{n}$ the $\sigma$-algebra generated by $(X_{1}, \dots, X_{n})$. We can bound 
\begin{align*}
\E_{x}[Y_{n+1}|\cF_n] & = \E_{x}[\1_{\tau>n+1}f(X_{n+1})|\cF_n]\\
& \leq \E_{x}[\1_{\tau>n} f(X_{n+1})|\cF_n]\\
& = \1_{\tau>n} \E_{x}[f(X_{n+1})|\cF_n]\\
& \leq \1_{\tau>n}(f(X_n)-\lambda_{1}).
\end{align*}
where the case $n=0$ uses that $x\notin K$. Taking expectation, we obtain 
\[
0 \leq \E_x[Y_{n+1}] \leq \E_x[Y_n]-\lambda_{1} \PP_x(\tau>n)
\leq \dots \leq \E_x[Y_0] - \lambda_{1} \sum_{k=0}^n\PP_x(\tau>k)
\]
and passing to the limit, we conclude that
\[
\E_x[\tau] = \sum_{k\geq 0}\PP_x(\tau>k) \leq \frac{f(x)}{\lambda_{1}}.
\]
\end{proof}

\begin{proof}[Proof of Lemma \ref{return-times}]
\underline{Case 1:  $(\mathbf{SD})$ is satisfied}. 

 Denote  by $(Y_{n})$ the Markov chain on $\R$ with i.i.d. increments given by $Z_{1}$ and set 
$$\tau^{(Y)} = \{\inf n\geq 1, \,Y_{n} \leq  R_{0}\} $$ the first return time of $(Y_{n})$ to $(-\infty,R_{0}]$. Using the dominance condition $(\mathbf{SD})$, one sees by induction on $n$ that for every $x\in \Omega$ such that $f(x)>R_{0}$, we have 
\[
\mathbb{P}_{x}(\tau\geq n)\geq  \mathbb{P}_{f(x)}(\tau^{(Y)}\geq n).
\]
Now for $n\geq 2$, $x\in K$, we can write
\begin{align*}
\PP_{x}(\tau \geq n) &= \int_{(R,+\infty)}\PP_{X_{1}}(\tau \geq n-1)\,d\PP_{x}(X_{1})\\
&\leq \int_{(R,+\infty)}\PP_{f(X_{1})}(\tau^{(Y)} \geq n-1)\,d\PP_{x}(X_{1})\\
&\leq \int_{(R,+\infty)}\PP_{R+Z_{0}}(\tau^{(Y)} \geq n-1)\,d\PP(Z_{0})
\end{align*} 
where the last inequality is obtained using assumption $(\mathbf{SD})$ and the fact that $t\mapsto \PP_{t}(\tau^{(Y)} \geq n-1)$ is increasing on $(R,+\infty)$. Observing that the right-hand side no longer depends on $x$ and summing over $n$ we get 
\begin{align*}
\sum_{n\geq 2}\sup_{x \in K}\mathbb{P}_{x}(\tau\geq n) &\leq  \int_{(R,+\infty)} \sum_{n\geq 2}\PP_{R+Z_{0}}(\tau^{(Y)} \geq n-1)\,d\PP(Z_{0})\\
&=\int_{(R,+\infty)} \E_{R+Z_{0}}(\tau^{(Y)})\,d\PP(Z_{0}).
\end{align*} 
However, Foster's criterion  (Lemma \ref{Foster}) gives for every $t> R$, the bound $\E_{t}(\tau^{(Y)})\leq (t-R)/\lambda_{1}$. Hence we may conclude that
\begin{align*}
\sum_{n\geq 2}\sup_{x \in K}\mathbb{P}_{x}(\tau\geq n) 
&\leq \frac{1}{\lambda_{1}}\int_{(R,+\infty)} Z_{0}\,d\PP(Z_{0})\\
&<\infty.
\end{align*}

\noindent \underline{Case 2:  $(\mathbf{D})$ and $(\mathbf{M_{\eta}})$ are satisfied from some $\eta>0$}.

We can assume $\eta\in (0,1)$.
Since $\PP_x(\tau\geq n)\leq\frac{1}{n^{1+\eta}}\E_x[\tau^{1+\eta}]$, it is enough to prove that the family of return times $(\tau \,|\,X_{0}=x)_{x\in K}$ is uniformly bounded in $L^{1+\eta}$. More generally, we show that for every $x\in \Omega$, 
\begin{align} \label{return-times-0}
\mathbb{E}_{x}(\tau^{1+\eta})\leq \left(\frac{f(x)}{\lambda_{1}}\right)^{1+\eta} +M\frac{f(x)}{\lambda_{1}} +M_{1}
\end{align}
where $M_{1}= 1+4\left((M^{\frac{1}{1+\eta}}+R_{0})^{1+\eta} +\frac{M(M+R_{0})}{\lambda_{1}}  \right)$.

We first bound the return time for a starting point oustide of $K$. On the one hand, notice that we already know by Foster's criterion that for every $x\in \Omega\smallsetminus K$, $\E_{x}(\tau)\leq \frac{f(x)}{\lambda_{1}}$.
In particular,  Jensen's inequality gives
\begin{align} \label{return-times-1}
\E_{x}(\tau^\eta)\leq \left(\frac{f(x)}{\lambda_{1}}\right)^\eta.
\end{align} 
On the other hand, the combination of $(\mathbf{D})$ and $(\mathbf{M_{\eta}})$ yields a strong negative drift for $f(X_{1})^{1+\eta}$ as we now explain. Observe that the inequality  
\begin{equation*}
(1+t)^{1+\eta}\leq 1+(1+\eta)t+\abs{t}^{1+\eta} 
\end{equation*}
where $-1\leq t \leq +\infty$, 
applied to $t=\frac{f(X_{1})-f(x)}{f(x)}$
gives
\[
f(X_{1})^{1+\eta} \leq f(x)^{1+\eta} + (1+\eta)\left(f(X_{1})-f(x)\right)f(x)^{\eta} + \abs{f(X_{1})-f(x)}^{1+\eta},
\]
Taking expectation, we obtain for every $x\in \Omega\smallsetminus K$, 
\begin{align} \label{return-times-2}
\E_{x}(f(X_{1})^{1+\eta}) \leq f(x)^{1+\eta} - (1+\eta)\lambda_{1} f(x)^{\eta} + M.
\end{align}

Inequalities (\ref{return-times-1}) and (\ref{return-times-2}) allow to  apply a general result of Tweedie \cite[Theorem 3(iii)]{Tweed83} (with $A=K$, $\psi(k)=(1+\eta)\lambda^{1+\eta}_{1}k^\eta-M$, $g(x)=f(x)^{1+\eta}$, $g_{1}(x)= (1+\eta)\lambda_{1} f(x)^{\eta} - M$) to obtain for all $x\in \Omega\smallsetminus K$,   
$$\E_{x}(\tau^{1+\eta})\leq \left(\frac{f(x)}{\lambda_{1}}\right)^{1+\eta} +M\E_{x}(\tau)$$
so by Foster's criterion again, for  $x\in \Omega\smallsetminus K$,
\[
\E_{x}(\tau^{1+\eta})\leq \left(\frac{f(x)}{\lambda_{1}}\right)^{1+\eta} +M\frac{f(x)}{\lambda_{1}}.
\]
The general bound announced in (\ref{return-times-0})  for every $x$ in $\Omega$ follows using the moment condition $(\mathbf{M_{\eta}})$.     

\end{proof}

\subsection{Conservation of mass} \label{Sec2.2}

 We show conclusion \ref{pro} in Theorem \ref{thD}:  random walks on $\Omega$ satisfying the stochastic dominance assumption $(\mathbf{SD})$,  or having a negative drift and increments uniformly bounded in $L^{1+\eta}$, have no escape of mass. 

\begin{proposition}\label{rec-law}
Suppose $(\Omega, f, (X_{n}))$   satisfy condition $(\mathbf{SD})$, or conditions $(\mathbf{D})$ and $(\mathbf{M}_{\eta})$ for some $\eta>0$. Then for all $\eps>0$, there exists $R> 0$ such that for  $x\in\Omega$,  $n\geq 1$,
\[
\PP_x[f(X_n)>R] \,<\, \eps + \frac{f(x)}{n\lambda_{1}}
\]
\end{proposition}

\begin{remark}
Choosing $\eps$ very small, we see that the first time  $n\geq1$ such that $\PP_x[f(X_n)>R] <1$ is bounded above by $\frac{f(x)}{\lambda_{1}(1-\eps)}\simeq \frac{f(x)}{\lambda_{1}}$. When $f(x)\to +\infty$, this bound is essentially optimal, as one can see by considering a deterministic walk by translation of $-\lambda_{1}$, for which the return time is equal to $\frac{f(x)-R}{\lambda_1}$.
\end{remark}

The key is to use Lemma \ref{return-times}  to  show  the following  renewal estimate.

\begin{lemma}[Renewal estimate]
\label{renewal}
Suppose $(\Omega, f, (X_{n}))$   satisfy condition $(\mathbf{SD})$, or conditions $(\mathbf{D})$ and $(\mathbf{M}_{\eta})$ for some $\eta>0$. Then for any $\alpha>0$, there exists $l\geq 0$,  such that for all $x\in K$, $n\geq l$, 
\[
\mathbb{P}_{x}(\exists i\in \{ n-l, \dots,n\}:  \, X_{i}\in K )> 1-\alpha.
\]
\end{lemma}

In words, this lemma states that for any interval of integers $I\subseteq \N$ of large length $l$, the walk starting from an arbitrary point in $ K$ has a very good chance to come back to $ K$ during $I$.

\begin{proof}
Let $\alpha>0$. 
According to Lemma \ref{return-times},  we may choose $l\geq 0$ such that
\[
\sum_{n\geq l}\sup_{x\in K}\PP_x(\tau>n) < \alpha.
\]
Let us check that $l$ satisfies the conclusion of the lemma. Given $n\geq l$, it is convenient to set $I_{n}= \{ n-l, \dots,n\}$. Assume that  $A_{n}\in \R^+$ is a constant such that for $x\in K$, $k \in \{l, \dots, n\}$,  we have 
\[
\mathbb{P}_{x}(\forall i\in I_{k},  \, X_{i}\notin K )\leq  A_{n}.
\]
We can write
\begin{align*}
\mathbb{P}_{x}(\forall i\in I_{n+1},  \,X_{i}\notin K )\,&\leq  \,\mathbb{P}_{x}(\forall i\in I_{n+1},  \,X_{i}\notin K ; \tau \leq n  ) + \mathbb{P}_{x}(\tau >n)\\
&\leq \sum_{k=1}^{n-l} \,\mathbb{P}_{x}(\forall i\in I_{n+1},  \,X_{i}\notin K ; \tau =k  ) + \mathbb{P}_{x}(\tau >n)\\
&\leq \sum_{k=1}^{n-l} \,A_{n}\mathbb{P}_{x}(\tau =k ) + \mathbb{P}_{x}(\tau >n)\\
&\leq A_{n} + \mathbb{P}_{x}(\tau >n).
\end{align*}
Hence, defining  a sequence $(A_{n})_{n\geq l}$ by 
\[
\left\{\begin{array}{l}
A_l=0\\
\forall n\geq l,\ A_{n+1} = A_n + \sup_{x \in K}\PP_x(\tau>n)
\end{array}\right.
\]
we can see by induction on $n$ using the above inequality that for all $x\in K$, $n\geq l$, 
\[
\mathbb{P}_{x}(\forall i\in I_{n},  \, X_{i}\notin K )\leq  A_{n} < \alpha.
\]
\end{proof}

\bigskip

Let us now prove the conservation of mass anounced above. 

\begin{proof}[Proof of Proposition \ref{rec-law}]

We first deal with the case where $x\in K$.  We denote by $l\geq 0$ the constant given by  Lemma \ref{renewal} for $\alpha=\eps/2$, then choose the constant $R>0$ large enough so that for every $x\in K$ and every $n\leq l$, 
\[
\mathbb{P}_{x}( f(X_{n})>R )< \eps/2.
\]
By Lemma \ref{renewal}, this yields for  $x\in K$, and arbitrary $n\geq l$ 
\begin{align*}
\mathbb{P}_{x}( f(X_{n})>R ) &<  \eps/2 + \mathbb{P}_{x}( f(X_{n})>R;  \,\,\exists i\in \{ n-l,\dots, n\},  X_{i}\in K  ) \\
&<\eps.
\end{align*}
To deal with the complementary case $x\in \Omega \smallsetminus K$, denote by $\tau$ the first return time to $K$ and notice that for $n\geq 0$, 
\begin{align*}
\mathbb{P}_{x}( f(X_{n})>R) &< \mathbb{P}_{x}( f(X_{n})>R;\,\tau \leq n)+  \mathbb{P}_{x}( \tau > n)  \\
&< \eps + \frac{\mathbb{E}_{x}[\tau]}{n}.
\end{align*}
In view of Foster's criterion recalled in Lemma \ref{Foster}, we have $\mathbb{E}_{x}[\tau] \leq  \frac{f(x)}{\lambda_{1}}$ as soon as $f(x)>R_{0}$. Hence, we obtain 
\begin{align*}
\PP_x \left(f(X_n)>R\right) < \eps +\frac{f(x)}{n\lambda_{1}}
\end{align*}
which finishes the proof of Proposition \ref{rec-law}
\end{proof}

\bigskip
 We conclude Section \ref{Sec2.2} by a counterexample to Proposition \ref{rec-law} when the walk $(\Omega, f, X_{n})$ is only assumed to satisfy $(\mathbf{D})$ and $(\mathbf{M}_{0})$.

\begin{example}[Escape of mass] \label{counterex-L1}
Let  $[0,\frac{1}{2}]\rightarrow [0,1), x\mapsto \alpha_{x}$ be a continuous map such that $\alpha^{-1}(0)=0$. We consider  the  Markov chain $(X_n)_{n\geq 0}$ on $\Omega=[0,+\infty)$ whose transition probability measures are given by
\[
\left\{\begin{array}{ll}
\mu_{0}=\delta_{0}   \\
\mu_{x}=(1-\alpha_{x})\delta_{x}+ \alpha_{x}\delta_{\alpha^{-1}_{x}} & \text{if $0<x\leq 1/2$}  \\
\mu_{x}=(2x-1)\delta_{0}+ (2-2x)\mu_{\frac{1}{2}} & \text{if $1/2\leq x\leq 1$}  \\
\mu_{x}=\delta_{x-1} & \text{if $x\geq  1$.} \\
\end{array}\right.
\]
Note that $\mu_{x}$ depends continously on $x$.  
Letting $f=\text{Id}$, the random walk $(X_n)_{n\geq 0}$  satisfies  
\[
\sup_{x\in \Omega} \mathbb{E}_{x}[|f(X_{1})-f(x)|]\leq 1
\quad\text{and}\quad
\sup_{x \in (1,+\infty)} \mathbb{E}_{x}[f(X_{1})-f(x)]= -1.
\]
However, given a fixed $x_{0}\in (0,1/2]$, one can choose the family $(\alpha_x)_{x\in[0,1/2]}$ so that for every $R>0$, there exists $n\geq0$ such that
\begin{align}\label{counterex-0}
\mathbb{P}_{x_{0}}(f(X_{n})>R)\geq 1/2.
\end{align} 
To see this, first observe that a trajectory with origin $x_{0} \in (0,\frac{1}{2}]$  stays at $x_{0}$ for some time, then jumps out of $[0,1]$ and comes  back to it with constant increments of $-1$ until reaching a point $x_{1}\in [0,1]$.  Then, if $x_{1}\in (0, \frac{1}{2}]$, the process repeats itself to give some point $x_{2}\in [0,1]$, and so on. Arguing step by step, we can choose the coefficients $\alpha_{x}$ such that the return points $x_0, x_{1},x_{2}, \dots$ form an infinite sequence in $(0,\frac{1}{2}]$, satisfy $x_{i}>x_{i+1}$, $(x_{i})_{i\geq 0}\to 0$, and so that the sequence $(\alpha_{x_{i}})_{i\geq0}$ decreases  to $0$ fast enough to have for every $i\geq1$, some $n_{i}\geq1$ for which
\[
\mathbb{P}_{x_0}(X_{n_{i}}=x_{i})\geq 1-1/i.
\]
Let us now check the property \eqref{counterex-0} announced above. Fix $R>0$.
For $i,k\geq0$, 
\begin{align}\label{counterex-1}
\mathbb{P}_{x_0}(X_{n_{i}+k}\leq R) &=  \mathbb{P}_{x_0}(X_{n_{i}+k}\leq R; X_{n_{i}}=x_{i})+ \mathbb{P}_{x_0}(X_{n_{i}+k}\leq R; X_{n_{i}}\neq x_{i}) \nonumber\\
&\leq  \mathbb{P}_{x_{i}}(X_{k}\leq R)+ 1/i.
\end{align}
Assume $i$ to be large enough so that $\alpha_{x_{i}}^{-1}>R$ and set $k_{i}=\lfloor \alpha^{-1}_{x_{i}} -R \rfloor$  the greatest integer smaller  than $\alpha^{-1}_{x_{i}} -R$.
We can then write
\begin{align}\label{counterex-2}
 \mathbb{P}_{x_{i}}(X_{k_{i}}\leq R)&= \mathbb{P}_{x_{i}}(X_{1}, \dots, X_{k_{i}}= x_{i})  \nonumber\\ 
& =(1-\alpha_{x_{i}})^{k_{i}}  \nonumber\\
&= e^{k_{i}\log(1-\alpha_{x_{i}})} 
\underset{i\to+\infty}{\longrightarrow} e^{-1}.
\end{align}
Equations (\ref{counterex-1}) and (\ref{counterex-2}) give that for large $i$, 
\[
\mathbb{P}_{x_0}(X_{n_{i}+k_{i}}\leq R)\leq 1/2
\]
as announced in (\ref{counterex-0}). 
\end{example}

\subsection{Recurrence of empirical measures} \label{Sec2.3}

We  show conclusion \ref{tra} in Theorem \ref{thD}:  random walks on $\Omega$ satisfying the stochastic dominance assumption ($\mathbf{SD}$), or having a negative drift and increments uniformly bounded in $L^{1+\eta}$,   must have positive recurrent trajectories.

\begin{proposition} \label{rec-traj}
Suppose $(\Omega, f, (X_{n}))$   satisfy condition $(\mathbf{SD})$, or conditions $(\mathbf{D})$ and $(\mathbf{M}_{\eta})$ for some $\eta>0$. For every $\eps>0$, there exists $R>0$ such that for  every $x\in \Omega$,  $\PP_{x}$-almost every $(X_{n})$, 
\[
\limsup_{n\to +\infty} \frac{1}{n} \card \{k\in \{1,\dots, n\}\ |\ f(X_{k})> R\} < \eps.
\]
\end{proposition}
\begin{proof}
We start  with some notation.
Let $\tau_{i}$ be  the $i$-th return time in $K$.  More precisely,   $\tau_{0}=0$ by convention, and for all $i\geq 1$, 
\[
\tau_{i}= \inf\{n> \tau_{i-1}\ |\ X_{n}\in K \}.
\]
We also let $Z_{i}= \tau_{i}-\tau_{i-1}$ be the length of the  $i$-th excursion and
\[
Z^{ext}_{i}=\card \{n \in \{ \tau_{i-1}, \dots,\tau_{i}-1\}\ |\ f(X_{n})> R \}
\]
the time spent outside of $\{f\leq R\}$ during the excursion. 

To prove Proposition \ref{rec-traj}, it is sufficient to show that if $R>0$ is large enough, then for any $x\in K$,  $\mathbb{P}_{x}$-almost surely,  one has for large $n\geq0$,
\begin{align}\label{rec-traj-1}
\sum_{i=1}^n Z^{ext}_{i} <\eps \sum_{i=1}^{n-1} Z_{i}.
\end{align}
We explain how to obtain (\ref{rec-traj-1}). The assumption  $(\mathbf{SD})$, or the conditions $(\mathbf{D})$ and $(\mathbf{M}_{\eta})$, both guarantee that 
\begin{align} \label{rec-traj-2}
\sup_{x\in K}\PP_{x}(Z^{ext}_{1}\neq 0) \underset{R\to +\infty}{\longrightarrow} 0.
\end{align}
Indeed, we can write for $x\in K$, 
\begin{align*}
\PP_{x}(Z^{ext}_{1}\neq 0) &\leq \PP_{x}(Z^{ext}_{1}\neq 0;\, \tau \leq N)+\PP_{x}(\tau\geq N)\\
&\leq \PP_{x}(\exists i\leq N, f(X_{i})>R;\, \tau \leq N)+\PP_{x}(\tau\geq N)\\
& \leq \sum_{i\leq N}  \PP_{x}(f(X_{i})>R)+\PP_{x}(\tau\geq N)\\
&\leq  \frac{N(R_{0}+NM)}{R}+\sup_{y\in K}\PP_{y}(\tau\geq N)\\
\end{align*}
where $M>0$ is a fixed uniform $L^1$-bound on the positive increment of the walk, i.e. such that $\mathbb{E}_{x}[\max\left(0, f(X_{1})-f(x)\right)]<M$.  The right-hand side of the inequality does not depend on $x\in K$.  In view of Lemma \ref{return-times}, we can choose $N$ large, then $R$ even larger, to have the right-hand side arbitrarily close to $0$, whence  (\ref{rec-traj-2}). 

Observing that $\E_{x}(Z^{ext}_{1})\leq \E_{x}(\tau \1_{Z^{ext}_{1}\neq 0})$ and  using the uniform integrability of the variables $(\tau\,|\,X_{0}=x)_{x\in K}$  given by Lemma \ref{return-times}, we deduce from (\ref{rec-traj-2}) that 
\[
\sup_{x\in K}\E_{x}(Z^{ext}_{1}) \underset{R\to+\infty}{\longrightarrow} 0.
\]
In particular, we can choose $R$ so that $\sup_{x\in K}\E_{x}(Z^{ext}_{1})<\eps/2$.
Fix $x\in K$.  Denote by $\mathcal{F}_{i}$ the  sub $\sigma$-algebra of $\Omega$ generated by $(X_{0}, \dots, X_{\tau_{i}})$. The random variables $Z^{ext}_{i}$, $i\geq 1$, statisfy for $n\geq 1$,
\[
\mathbb{P}_{x} (Z^{ext}_{i}\geq n \,|\, \mathcal{F}_{i-1}) \leq \sup_{y\in K}\PP_{y}(\tau\geq n).
\]
As the right-hand side of the inequality is summable by Lemma \ref{return-times},  we may apply  Kolmogorov's law of large numbers for dependent variables  \cite[Theorem A.6]{BQRW} and obtain that $\mathbb{P}_{x} $-almost surely, 
\[
\frac{1}{n}\sum_{i=1}^n Z^{ext}_{i}  = \frac{1}{n}\sum_{i=1}^n \mathbb{E}_{x}[Z^{ext}_{i} \,|\, \mathcal{F}_{i-1}] +o(1).
\]
Moreover, by our choice for $R$, we have   $\mathbb{E}_{x}[Z^{ext}_{i} \,|\, \mathcal{F}_{i-1}] < \eps/2$.
We infer that for large $n\geq0$, 
\[
\frac{1}{n}\sum_{i=1}^n Z^{ext}_{i}  <  \eps\,\frac{1}{n}\sum_{i=1}^{n-1} Z_{i}
\]
keeping in mind that by definition $Z_{i}\geq 1$. This concludes the proof of inequality (\ref{rec-traj-1}), whence the proposition. 
\end{proof}

We conclude with a counterexample to Proposition \ref{rec-law} when the walk $(\Omega, f, X_{n})$ is only assumed to satisfy $(\mathbf{D})$ and $(\mathbf{M}_{0})$. 

\begin{example}[Escape of empirical measures] \label{counterex2-L1}
Given a real number $r\in \R$, we denote by $\lfloor r \rfloor$ the largest integer $k$ such that $k\leq r$.
Fix a sequence $(x_{i})_{n\geq 3}\in (0,1)^{\N_{\geq 3}}$ decreasing to $0$, and consider the Markov chain $(X_{n})_{n\geq 0}$ on $\Omega=[0,+\infty)$ whose transition probability measures are given by 
\[
\left\{\begin{array}{ll}
\mu_0 = \delta_0 & \\
\mu_{x_{i}}=(1-\frac{1}{i\log i})\delta_{x_{i+1}}+ \frac{1}{i\log i}\delta_{x_{i+1}+ \lfloor i \sqrt{\log i}\rfloor} & \text{ for $i\geq 3$}\\
\mu_{x}=\delta_{x-1} & \text{ if  $x> 1$}\\
x\mapsto \mu_{x}\ \text{affine on each segment}\ [x_{i+1},x_i]. &
\end{array}\right.
\]
The chain $(X_{n})_{n\geq 0}$ has a continuous family of transition probability measures and satisfies, for $f=\text{Id}$,
\[
\sup_{x\in \Omega} \mathbb{E}_{x}[|f(X_{1})-f(x)|]\leq 3
\quad\text{and}\quad
\sup_{x \in (1,+\infty)} \mathbb{E}_{x}[f(X_{1})-f(x)]= -1.
\]
However, for $\mathbb{P}_{x_{3}}$-almost every trajectory $(X_{n})_{n\geq 0}$, for every $\eps, R>0$, there exists $n_{0}\geq 1$ such that 
\begin{align}\label{counterex2}
\frac{1}{n_{0}} \card \{k\in \{1,\dots, n_{0}\}\ |\ f(X_{k})\leq R\}< \eps.
\end{align}
To see this, denote by $\tau_{j}$ the $j$-th return time to $[0,1]$. In other words,    $\tau_{0}=0$ by convention, and for all $j\geq 1$, 
\[
\tau_{j}= \inf\{n> \tau_{j-1}\ |\ X_{n}\in [0,1]\}.
\]
The random variables $(\tau_{j}-\tau_{j-1})_{j\geq 1}$ varying under $\mathbb{P}_{x_{3}}$  are independent  of  respective laws
\begin{align*}
(1-\frac{1}{(j+2)\log (j+2)})\delta_{1}+ \frac{1}{(j+2)\log (j+2)}\delta_{\lfloor (j+2) \sqrt{\log (j+2)}\rfloor +1 } \tag{$j\geq 1$}
\end{align*}
As the sequence $(\frac{1}{n\log n})_{n\geq 3}$ is not summable, the converse of the Borel-Cantelli lemma yields that  $\mathbb{P}_{x_{3}}$-almost surely we have $\limsup_{j\to +\infty}\frac{\tau_{j}-\tau_{j-1}}{j}=+\infty$.
In particular, given $\eps, R>0$, there exists $j_{0}\geq 1$ such that 
\[
j_{0}<\frac{\eps}{R+1}\,\tau_{j_{0}}.
\] 
Setting $n_{0}=\tau_{j_0}$, we get 
\[
\card \{k\in \{1,\dots, n_{0}\}\ |\ X_{k}\leq 1\}< \frac{\eps}{R+1} n_{0}.
\]
Finally, observing that for any $k\geq 0$ such that $X_{k}\leq R$ there exists a time $j\in \{k, \dots, k+\lfloor R\rfloor\}$ for which $X_{j}\in [0,1]$, we conclude 
\[
\card \{k\in \{1,\dots, n_{0}\}\ |\ X_{k}\leq R\}< \eps n_{0}
\]
as announced in (\ref{counterex2}).
\end{example}

\section{Recurrence criterion for  homogeneous chains} 
\label{Sec3}

 We apply the results of the previous section to the particular case of homogeneous random walks and obtain a handy criterion to check that the mass or the empirical measures of the walk do not accumulate near a given subset (Proposition \ref{cd-rec} below).
This criterion will be applied in two different contexts, first to prove non-escape at infinity, and later to derive non-accumulation on invariant homogeneous subspaces. 
 
 \subsection{The controlled drift condition}
 In the whole section we fix $\Omega$ a locally compact topological space, $\Gamma$ a topological semigroup acting continously on $\Omega$, and $\mu$ a Borel probability measure on $\Gamma$. We call $\mu$-walk on $\Omega$ the Markov chain on $\Omega$ whose law of transition at $x\in \Omega$ is  $\mu\ast \delta_{x}$, image of $\mu$ under $g\mapsto gx$. 
We denote by $\overline{\Omega}=\Omega \cup \{\infty\}$ the Alexandroff compactification of $\Omega$, and write  $\mathscr{P}(\overline{\Omega})$ the collection of its subsets. Given $\cS \subseteq \mathscr{P}(\overline{\Omega})\times \mathscr{P}(\overline{\Omega})$, we have a decomposition  $\cS=\cup_{Y\in p_{1}(\cS)}\{Y\}\times \cS_{Y}$  where $p_{1}$ is the first coordinate projection map and  $\cS_{Y}\subseteq\mathscr{P}(\overline{\Omega})$ is the fibre above $Y$ in $\cS$.  Finally, given a finite-dimensional real  vector space $V$, we set $\SL^{\pm}(V)=\{g\in\GL(V)\ |\ \det g=\pm 1\}$.

\begin{definition}\label{cd} 
Let $\cS \subseteq \mathscr{P}(\overline{\Omega})\times \mathscr{P}(\overline{\Omega})$, $\cA \subseteq \mathscr{P}(\Omega)$. We say that the $\mu$-walk has a \emph{controlled drift away from $\cS$ on $\cA$} if:
\begin{itemize}
\item
There exists a linear representation $\Phi : \Gamma\rightarrow\SL^{\pm}(V)$ on a finite-dimensional normed real vector space $V$ such that
\[
\int_{G}\log \|\Phi(g)\|\, d\mu(g)<+\infty.
\]
 \item
 There exist constants $\lambda,C>0$   such that for any $\alpha>0$, there exists $N\geq 1$ such that for $n\geq N$, for any $Y\in p_{1}(\cS)$, any compact subset $K\subseteq \Omega$ such that $K\cap Y=\varnothing$, any $Q\in \cA$, there is $Y'\in \cS_{Y}$, and a measurable function $f:\Omega\to[0,+\infty]$ which is bounded on $K$, whose upper level sets are neighborhoods of $Y'$, and which satisfies
\begin{enumerate}[label=(\roman*)]
\item \label{pdec} $\exists\, T\in [0,+\infty)$, $\forall x\in Q\cap f^{-1}[T,+\infty)$, 
\begin{align*} \tag{probable decrease}
\mu^{*n}(g\in G\ |\ f(gx)\leq f(x)-n\lambda) \geq 1-\alpha 
\end{align*}
\item \label{cvar} $\forall x\in f^{-1}[0,+\infty)$, $\forall g\in \Gamma$,  
\begin{align*} \tag{control of variations}
f(gx)-f(x) \leq C\log \norm{\Phi (g)}.
\end{align*}
\end{enumerate}
\end{itemize}
\end{definition}

\begin{remark}
1) Strictly speaking, the definition expresses a drift away from some $Y'\in \cS_{Y}$ which could have nothing to do with $Y$. However, in all the applications below, we will always have that $Y\subseteq Y'$ for all $Y'\in \cS_{Y}$, so in particular, the walk does drift away from $Y$ too. We choose to sum it up using the terminology of drift away from $\cS$. 

 3) The condition \ref{cvar} implies that $f^{-1}[0,+\infty)$ is $\Gamma$-invariant. 

 2) If $\cA=\{\Omega\}$  then we simply say that the $\mu$-walk has a \emph{controlled drift away from $\cS$}.  If additionally,  $\cS=\{\infty, \infty\}$,  we talk about \emph{controlled drift away from infinity}. Note that $f$ is proper in this case: its sublevel sets have compact closure in $\Omega$.  

4) The map $\Phi$ will later be referred to as the \emph{control function} for the drift. Indeed, its role is to bound the growth of $f$ in the unlikely event its value goes up along the walk. 
\end{remark}

To get more familiar with the notion of controlled drift, we start by exploring some stability properties. Given $\cS_{0}, \cS_{1}\subseteq \mathscr{P}(\overline{\Omega})\times \mathscr{P}(\overline{\Omega})$, we denote by $\cS_{0} \vee \cS_{1}$  the set of pairs $(Y, Y'_{0}\cup Y'_{1})$ where $Y\in p_{1}(\cS_{0}) \cap p_{1}(\cS_{1})$, $Y'_{0}\in \cS_{0,Y}$, $Y'_{1}\in \cS_{1,Y}$.

\begin{lemma}
 \label{drift-union} 
\begin{itemize}
\item[a)] If the $\mu$-walk on $\Omega$ has a controlled drift away from both $\cS_{0}, \cS_{1}$, then it has a controlled drift away from  $\cS_{0} \vee \cS_{1}$. 
\item[b)] \label{du2} If the $\mu$-walk on $\Omega$ has a controlled drift away from infinity and a controlled drift away from  $\cS$ on compact sets of $\Omega$, then it has a controlled drift away from  $\cS$. 
\end{itemize}
\end{lemma}

We explain \emph{a)}. The proof of the second  assertion is much more involved and postponed to Section \ref{Sec5} where it will be used.

\begin{proof}[Proof of Lemma \ref{drift-union} a)]

We  respectively index by $0$ and $1$ the notations referring to $\cS_{0}$, $\cS_{1}$. Let  $C_{i}, \lambda_{i}$ be the constants given by the controlled drift assumption \ref{cd} for $\cS_{i}$.  Set $C=\max(C_{0}, C_{1})$, $\lambda=\min(\lambda_{0}, \lambda_{1})$. Let $\alpha>0$, then $N_{i}$ as in \ref{cd}. Set $N=\max(N_{0}, N_{1})$, $n\geq N$. Let $Y\in p_{1}(\cS_{0}) \cap p_{1}(\cS_{1})$ and consider a compact set $K\subset \Omega\smallsetminus Y$. 
Let $Y'_{i}, f_{i}, T_{i}$ be given by \ref{cd}. We set $Y'=Y'_{0}\cup Y'_{1}$, $$f=\max(f_{0}, f_{1})$$
 and claim that the map $f$ satisfies all the requirements in Definition \ref{cd} (with $2\alpha$ instead of $\alpha$). Indeed, all the conditions are obvious except the one entitled ``probable decrease''.   To check it, we choose $T\geq \max(T_{0}, T_{1})$ and large enough so that  
$$\mu^{*n}(g\in \Gamma \,|\, T_{0}+C_{0}\Phi_{0}(g)\leq T-n\lambda) \geq 1-\alpha$$
and the same holds when we replace $0$ by $1$. 
Now let $x\in f^{-1}[T,+\infty)$. If $f_{1}(x)=f(x)$, then 
$$ \mu^{*n}(g\in \Gamma \,|\, f_{1}(gx)\leq f(x)-n\lambda_{1}) \geq 1-\alpha$$
On the other hand, distinguishing the cases $f_{0}(x)<T_{0}$ and $f_{0}(x)\geq T_{0}$, the condition on $T$ guarantees that 
$$ \mu^{*n}(g\in \Gamma \,|\, f_{0}(gx)\leq f(x)-n\lambda) \geq 1-\alpha$$
so combining these inequalities 
$$ \mu^{*n}(g\in \Gamma \,|\, f(gx)\leq f(x)-n\lambda) \geq 1-2\alpha.$$
The same holds if $f_{0}(x)=f(x)$ by symmetry.

\end{proof}

 \subsection{Application to recurrence}

The proposition below states that a walk that drifts away from $\cS$ does not accumulate on $\cS$.
In particular, when $\cS=\{\infty, \infty\}$, the conclusion exactly means that the walk has neither escape of mass nor escape of empirical measures at infinity. 
 
 \begin{proposition} \label{cd-rec}
 Let $(\Omega, G, \mu)$ be as above and  $\cS \subseteq \mathscr{P}(\overline{\Omega})\times \mathscr{P}(\overline{\Omega})$. Assume that the $\mu$-walk on $\Omega$ has a controlled drift away from $\cS$. Then for every $Y\in p_{1}(\cS)$, every compact set  $K\subseteq \Omega\smallsetminus Y$, there is $Y'\in \cS_{Y}$ such that for every $\eps>0$, for some neighborhood $O'$ of $Y'$, for all $x\in K$, 
  
 \begin{enumerate}[label=(\roman*)]
\item \label{cd-rec1} for every $n\geq 0$, $(\mu^{*n}*\delta_x)(O')<\eps$;
\item  \label{cd-rec2} for $\mu^{\otimes\N^*}$-almost every instructions $(g_i)_{i\geq 1}$,
\[
\limsup_{n\to+\infty}\frac{1}{n}\card\{k\in\{1,\dots,n\}\ |\ g_k\dots g_1x\in O'\}< \eps.
\]
\end{enumerate}
 \end{proposition}
 
 The proof can be summarized as follows. If $\alpha>0$ is small enough, $n\geq N$ is large enough, then for every $Y\in p_{1}(\cS)$, every compact set $K\subseteq \Omega\smallsetminus Y$, the function $f$ given by the drift property \ref{cd} satisfies the assumptions of Lemma \ref{crit-SD} for the $\mu^{*n}$-random walk on $f^{-1}[0,+\infty)$ for some $n\in\N^*$, whence the stochastic dominance condition $(\mathbf{SD})$. Theorem \ref{thD} then applies and leads to non-accumulation results for the $\mu^{\ast n}$-walk on $f^{-1}[0,+\infty)$ which can be extended to the $\mu$-walk.

We begin with a lemma that will allow to check condition $\emph{3.}$ in Lemma \ref{crit-SD}.

\begin{lemma}
\label{standard}
Let  $\Phi: \Gamma\rightarrow\SL^{\pm}(V)$ be a linear representation of $\Gamma$ on a finite-dimensional normed vector space $V$ such that $$\int_{G}\log \|\Phi(g)\|\, d\mu(g)<+\infty $$ Then for all $\eta>0$, there exists $\alpha>0, N'\geq 0$ such that for all $n\geq N'$, if $Z_n'$ denotes the standard realisation of $Z_n=\log\norm{\Phi (g)}$ for $\mu^{*n}$, then
\[
\E[Z_n'\1_{[0,\alpha]}] \leq n \eta.
\]
\end{lemma}
\begin{proof}

The law of large numbers \cite[Lemma 4.27]{BQRW} guarantees that there exists a constant $l\geq 0$ such that the sequence of real random variables $(\frac{1}{n}Z_{n})_{n\geq1}$ converges to $l$ in $L^1$. This convergence also holds for  $(\frac{1}{n}Z'_{n})_{n\geq 1}$ as it has the same law. Now let $\alpha>0$, $n\geq 1$. 
\begin{align*}
\frac{1}{n}\E[Z_n'\1_{[0,\alpha]}] & \leq \alpha l + \E[\abs{\frac{1}{n}Z'_{n} -l}]\\
& \leq \alpha (l+1)
\end{align*}
if $n$ is large enough. The result follows if we specify $\alpha=\frac{\eta}{l+1}$ above.
\end{proof}

We now apply the results of Section~\ref{Sec2} with the drift function $f$ constructed above to derive \ref{thA'}.

\begin{lemma}\label{SDmun}
Assume that the $\mu$-walk on $\Omega$ has a controlled drift away from $\cS$. Then there exists $\alpha_{0}>0$, $n_{0}\geq N$ such that for any $Y\in p_{1}(S)$, any compact $K\subseteq \Omega\smallsetminus Y$, any map $f$ as in Definition \ref{cd}, the  $\mu^{\ast n_{0}}$-walk on $f^{-1}[0,+\infty)$ satisfies condition $(\mathbf{SD})$ with drift function $f$.
\end{lemma}
\begin{proof}
Let $\lambda, C>0$ be given by Definition~\ref{cd}.
Choose a small $\alpha_{0}\in (0,1/2)$ as in Lemma \ref{standard} for $\eta=\frac{\lambda}{2C}$.  Fix an integer $n_{0}\geq \max(N,N')$ where $N, N'$ are the constants given respectively by Definition~\ref{cd} and Lemma \ref{standard} for our choice of $\alpha$.
Consider $Y\in p_{1}(\cS)$, a compact set $K\subseteq \Omega\smallsetminus Y$, and a map $f$ as in Definition~\ref{cd}.

We   show that the $\mu^{\ast n_{0}}$-random walk on $f^{-1}[0,+\infty)$ satisfies the stochastic dominance condition  $(\mathbf{SD})$. To do so, we only need to check that the assumptions of Lemma~\ref{crit-SD} do hold in the present context.  

\begin{enumerate}
\item By Definition \ref{cd}, there exists $T>0$ such that for $x\in f^{-1}[T,+\infty)$, 
\[
\mu^{*n_{0}}(g\in \Gamma\ |\ f(gx) \leq f(x)-n_{0}\lambda) \geq 1-\alpha_{0}.
\]
\item For every $x\in f^{-1}[0,+\infty)$, $g\in \Gamma$,  we have $f(gx)-f(x) \leq C\log \|\Phi(g)\|$.
In particular, denoting by $Z$ the variable $C\log \|\Phi(g)\|$ where $g$ varies with respect to $\mu^{\ast n_{0}}$, the law of $f(gx)-f(x)$ for $\mu^{\ast n_{0}}$ is stochastically dominated by $Z$.  
\item Lemma \ref{standard} and our choice of constants imply that, writing $Z'$ the standard realisation of $Z$,  
\[
\E[Z'\1_{[0,\alpha_{0}]}]\leq \frac{n_{0}\lambda }{2}.
\]
\end{enumerate}
Hence the conditions of Lemma~\ref{crit-SD} are satisfied, which implies that $(\mathbf{SD})$ holds for the $\mu^{\ast n_{0}}$-walk on $f^{-1}[0,+\infty)$. 
\end{proof}

\begin{proof}[Proof of Proposition \ref{cd-rec}]
Fix $\alpha_{0}>0$, $n_{0}\geq N$  as in Lemma \ref{SDmun},   $Y\in p_{1}(\Omega)$, a compact $K\subseteq \Omega\smallsetminus Y$, and let $Y'\in \cS_{Y}$, $f:\Omega\rightarrow [0,+\infty]$ be given by the drift condition  \ref{cd}.
Let us check the conclusions of the proposition with $O'=\{f> M\}$ for some sufficiently large $M$.
Lemma \ref{SDmun} combined with Theorem \ref{thD} already yields the result for the $\mu^{*n_{0}}$-walk: for any $\eps>0$, there is $R>0$ such that 
for all $x\in K$,

\begin{enumerate}[label=(\roman*)]
\item for every $q\geq 0$, $\mu^{\ast qn_{0}}(g\in \Gamma\ |\ f(gx) > R) < \eps$;
\item for $\mu^{\otimes\N^*}$-almost every trajectory $(g_i)_{i\geq 1}$,
\[
\limsup_{n\to+\infty}\frac{1}{q}\card\{k\in\{1,\dots, q\}\ |\ f(g_{kn_{0}}\dots g_{1}x)> R\} <\eps.
\]
\end{enumerate}

To conclude, we just check that \ref{cd-rec1} and  \ref{cd-rec2} also hold for $\mu$ (i.e. if $n_{0}=1$) up to increasing $R$ and doubling $\eps$.
\bigskip

For \ref{cd-rec1},  choose $R_1> 0$ such that for every $1\leq m<n_{0}$,
\[
\mu^{\ast m}(g\in G\ |\ \log  \norm{\Phi(g)}> R_1) \leq\eps.
\]
Let $R_2=R+CR_1$ and $n\geq0$. Writing  $ n =qn_{0}+m$ for some integers $q, m\geq 0$ with $m<n_{0}$,  we obtain for every $x\in K$
\begin{align*}
\mu^{\ast n}(g\in G\ |\ f(gx)>R_{2}) & \leq \mu^{\ast qn_{0}}(g\in G\ |\ f(gx)>R)\\
& \qquad + \mu^{\ast m}(g\in G\ |\ \log\norm{\Phi(g)}>R_{1}) \\
& < 2\eps.
\end{align*}

\bigskip
For \ref{cd-rec2}, the argument is very similar.
Given $R_{3}>0$, set for $q\geq 1$
\[ 
Y_q = \1_{\{\exists n\in \{qn_{0}+1, \dots, qn_{0}+n_{0}\}:\ \log\norm{\Phi(g_n\dots g_{qn_{0}+1})}>R_3\}}.
\]
The random variables $Y_q$ are independent, identically distributed, and  satisfy  $\E[Y_q]< \eta$ if $R_{3}$ is chosen large enough. 
Therefore, by the classical law of large numbers,
\[
\limsup_{q\to+\infty} \frac{1}{q}\card\{k\in\{1,\dots,q\}\ |\ \exists\, kn_{0}<n\leq n_{0}(k+1):\log\norm{\Phi(g_n\dots g_{kn_{0}+1})}>R_3\} < \eps.
\]
Setting $R_{4}=R+CR_{3}$, we get for $\mu^{\otimes \N^\ast}$-almost every $(g_{i})_{i\geq 1}\in G^{\N^\ast}$,
\[
\limsup_{n\to+\infty}\frac{1}{n}\card\{k\in\{1,\dots,n\}\ |\ f(g_{k}\dots g_{1}x)>R_4\} < 2\eps
\]
which finishes the proof of the proposition.
\end{proof}

\section{Non-escape at infinity}
\label{Sec4}

We now apply the results of the previous section to show that random walks on finite-volume homogeneous spaces have neither escape of mass nor escape of empirical measures at infinity.
In what follows, $G$ denotes a real Lie group,  $\g$ its Lie algebra, $\kr$ its amenable radical, $\ks=\g/\kr$ the largest  quotient of $\g$ which is semisimple without compact factors, and we denote by $\Ad_{\ks}:G\to\Aut\ks$ the adjoint action of $G$ on $\ks$.
Given a probability measure $\mu$ on $G$, we let  $H_{\ks}\subset\Aut\ks$ be the real algebraic group generated by the support of $(\text{Ad}_{\mathfrak{s}})_{\ast}\mu$ and call $H^{nc}_{\ks}$ its non-compact part, i.e. the smallest normal algebraic subgroup of $H_{\ks}$ such that $H_{\ks}/H^{nc}_{\ks}$ is compact.

Our goal is to show the following theorem, which is due to Benoist and Quint \cite[Theorem~1.1]{bq_fv} in the case where $\mu$ has a finite exponential moment on $\ks$.

\begin{named}{Theorem A$^\prime$}\emph{(Recurrence of semisimple random walks)}
\label{thA'} 
Let $G$ be a real Lie group, $\Lambda$  a lattice in G, and $\Omega=G/\Lambda$.
Assume that $\mu$ is a probability measure on $G$ such that
\begin{enumerate}
\item $\mu$ has a finite first moment on $\ks$: $\int_G \log\norm{\Ad_{\ks} g}\ d\mu(g) < +\infty$
;
\item $H^{nc}_{\ks}$ is semisimple.
\end{enumerate}
Then, for every compact set $K\subset\Omega$ and every $\eps>0$, there exists a compact set $K'\subset \Omega$ such that for every $x\in K$,
\begin{enumerate}[label=(\roman*)]
\item for every $n\geq 0$, $(\mu^{*n}*\delta_x)( K') >1- \eps$;
\item for $\mu^{\otimes\N^*}$-almost every  instructions $(g_i)_{i\geq 1}$,
\[
\liminf_{n\to+\infty}\frac{1}{n}\card\{k\in\{1,\dots,n\}\ |\ g_k\dots g_1x \in K'\} > 1-\eps.
\]
\end{enumerate}

\end{named}

\begin{remark}
In particular,  every Radon $\mu$-stationary measure on $\Omega$ must have finite mass. 
\end{remark}

By Proposition \ref{cd-rec}, the proof of \ref{thA'} reduces to showing that the $\mu$-walk on $\Omega$ has a controlled drift away from infinity.
We first check that this is  the case if $\Omega=\SL_d(\R)/\SL_d(\Z)$ and then deduce the result for general quotients $\Omega=G/\Lambda$.  

\subsection{The Benoist-Quint drift function on $\SL_d(\R)/\SL_d(\Z)$} \label{Sec4.1}

We prove \ref{thA'}  in the particular case where $G=\SL_{d}(\R)$, $\Lambda=\SL_{d}(\Z)$ for some $d\geq 2$.
The drift functions we use were defined (in an exponential form) by Benoist and Quint in \cite{bq_fv}. They are inspired by former works of Eskin, Margulis, Mozes \cite{EskMar04,EskMarMoz98}. We begin by recalling their construction.

\bigskip

 Fix a probability measure $\mu$ on $\SL_d(\R)$, and let $H$ be the algebraic group generated by the support of $\mu$.
 We assume that the non-compact part $H^{nc}$ of $H$ is semisimple, or equivalently that $H$ is reductive with compact center. Consider the representation of $H$ on  the exterior algebra $V=\wedge^*\R^d$ and decompose it into its isotypical components for the induced action of $H^{nc}$:
 \[
V = \oplus_{j\in J} V^{(j)}.
\]
Alternatively, denoting by $H^{\circ}$ the identity component of $H$ in the Zariski topology, each $V^{(j)}$ is obtained by summing together  all the  irreducible subrepresentations of $H^\circ$  on $V$ with a given highest weight. The action of $H$ on $V$ then induces an action of the group of connected components $F=H/H^{\circ}$ on the set $\{V^{(j)}, j\in J\}$ which stabilizes the subspace $V^{H^{nc}}$ of $H^{nc}$-invariant vectors.  By the law of large numbers \cite[Theorem 10.9, Corollary 10.12]{BQRW}, we can associate to each representation $V^{(j)}$ a Lyapunov exponent $\lambda_1^{(j)}\in \R$, that quantifies the exponential growth rate of the walk on $V^{(j)}$ and only depends on the $F$-orbit of $V^{(j)}$. Moreover, by semisimplicity of $H^{nc}$, we have $\lambda^{(j)}_{1}\geq 0$, with  $\lambda_1^{(j)}=0$ if and and only if $V^{(j)}=V^{H^{nc}}$ \cite[Theorem 10.9 (f)]{BQRW}.

Fix $H^c\subseteq H$ a maximal compact subgroup of $H$ and some $H^c$-invariant Euclidean norm on $\R^d$.
This induces a natural $H^c$-invariant Euclidean norm on the exterior algebra $V$. 
Given $v\in V$, write $v=\sum_{j\in J}v^{(j)}$ according to the above decomposition of $V$, and define 
\[
\abs{v} = \max_{j\in J} \norm{v^{(j)}}^{\frac{1}{\lambda_1^{(j)}}}
\]
with the convention that $\frac{1}{0}=+\infty$ and
\[
\norm{v^{(j)}}^{+\infty} = \left\{ \begin{array}{ll}
0 & \mbox{if}\ \norm{v^{(j)}} < 1\\
+\infty & \mbox{if}\ \norm{v^{(j)}} \geq 1.
\end{array}
\right.
\]

\bigskip
The following lemma expresses that the $\mu$-walk on $V$ uniformly expands the quasi-norm of vectors  $v$ whose $H^{nc}$-invariant component is not too weighty.

\begin{lemma}
\label{furst}
Given $k\geq 1$, $\lambda\in (0,1)$, $\alpha>0$, there exists $n_{0}\geq 0$ such that for every $n\geq n_0$, for every finite family $v_1,\dots,v_k$ of non-zero vectors in $\wedge^*\R^d$ satisfying $0<\abs{v_l}<+\infty$ for each $l$, one has 
\[
\mu^{*n}\left(g\in H\ \left|\ \min_{1\leq l\leq k}\log\frac{\abs{g v_l}}{\abs{v_l}}\geq n\lambda\right.\right)\geq 1-\alpha.
\]
\end{lemma}
\begin{proof}
Up to replacing $\alpha$ by $\alpha/k$, we may assume that $k=1$.  For $v\in V$ such that $0<\abs{v}<+\infty$, write
\(
v = \sum_{j\in J} v^{(j)},
\)
and let
\[
J_v=\{j\in J\ |\ \abs{v^{(j)}}>0\}.
\]

Then, for every $g\in H$, one can bound
\begin{align}\label{ineq-min-||}
\min_{j\in J_v} \log \frac{\abs{gv^{(j)}}}{\abs{v^{(j)}}}
\leq \log \frac{\abs{gv}}{\abs{v}}
\leq \max_{j\in J_v} \log \frac{\abs{gv^{(j)}}}{\abs{v^{(j)}}}. 
\end{align}
For the left inequality, just consider $j$ such that $\abs{v}=\abs{v^{(j)}}$, and for the right inequality, choose $j$ for which $\abs{gv}=\abs{gv^{(j)}}$.

By the law of large numbers for product of random matrices \cite[Theorem~4.28 (b)]{BQRW}, on each  representation $V^{(j)}$ such that $\lambda_1^{(j)}>0$, the sequence of functions
\[
(g_i)_{i\geq 1}\mapsto \frac{1}{n}\log\frac{\abs{g_n\dots g_1 v^{(j)}}}{\abs{v^{(j)}}}
\]
converges to $1$ in $L^1(G^{\N^*}, \mu^{\otimes \N^*})$, uniformly over all non-zero $v^{(j)}\in V^{(j)}$.
Combining this fact  and the left inequality of \eqref{ineq-min-||} yields  the lemma.
\end{proof}

To define the Benoist-Quint drift function on the set $\Omega =\SL_{d}(\R)/\SL_{d}(\Z)$ of unimodular lattices of $\R^d$, we need to enhance some convexity property of the above quasi-norm on $V$.
For that, we make the following definition.

\begin{definition}[Corrected quasi-norm]
Given a large parameter $A> 0$ and $i\in\{1, \dots, d-1\}$, define a corrected version $\abs{\cdot}_{A}$ of the quasi-norm $\abs{\cdot}$ by
\[
\forall v\in\wedge^i\R^d,\quad \abs{v}_A =\Abs{e^{Ai(d-i)}v}
\]
\end{definition}

We can also view $\|.\|$, $\abs{\cdot}$ and $\abs{\cdot}_{A}$ as  functions on discrete subgroups of $\R^d$, by identifying $\Delta=\Z a_1\oplus\dots\oplus\Z a_i$ with $\pm a_1\wedge\dots\wedge a_i \in \wedge^{i}\R^d$. This allows to define a function $\phi_A$ with values in $[-\infty, +\infty]$ by setting:
\[
\phi_A(\Delta) = \log\frac{1}{\abs{\Delta}_A}\ \ \mbox{if}\ \text{rank\,} \Delta\notin \{0,d\}
\quad\mbox{and}\quad
\phi_A(\Delta)=0 \\\ \mbox{otherwise.}
\]

\begin{remark}
A discrete subgroup $\Delta\subseteq \R^d$ of rank $i\in \{1, \dots, d-1\}$ satisfies  $\phi_A(\Delta)=+\infty$ if and only if $\Delta\in V^{H^{nc}}$ and $\norm{\Delta}<e^{-Ai(d-i)}$. Since $H$ acts by isometries on  $V^{H^{nc}}$, it follows that the set $\{\phi_{A}=+\infty\}$ is $H$-invariant.
\end{remark}

\begin{definition}[Benoist-Quint function]

For $A> 0$, define a function $f_A$ on $\Omega$ with values in $[0,+\infty]$ by
\[
f_A(x) = \max_{\Delta\leq x} \phi_A(\Delta).
\]
\end{definition}

A few comments are in order. 
First of all, the notation $\Delta\leq x$ means that $\Delta$ varies in the set of  subgroups of the lattice $x$.
As $\Delta$ can be chosen of rank $0$ or $d$, our conventions for $\phi_{A}$ impose that $f_{A}\geq0$. 

Note also that the maximum defining $f_A$ is well defined.
Indeed, one can observe that if $\Delta\leq x$ satisfies $\phi_{A}(\Delta)>0$ then $\text{rank}(\Delta)\notin \{0, d\}$ and $\|\Delta^{(j)}\|< 1$ for each $j\in J$.
Since $\wedge^{*}x$ is a discrete subset of $V=\wedge^*\R^d$, the set of subgroups $\Delta\leq x$ such that $\phi_A(\Delta)>0$ is finite. 

Finally, observe that the set $\Omega_{A}:=\{f_{A}<+\infty\}$ is $H$-invariant, because this is the case for $\{\phi_{A}=+\infty\}$. Moreover, for every compact subset $K\subset\Omega$, the norm $\|\Delta\|$ of non-trivial subgroups of lattices $x\in K$ is uniformly bounded away from $0$ as $x$ varies in $K$.  This implies that $K\subset\Omega_A$ for large enough $A$.

\begin{proposition}[Drift for the Benoist-Quint function on $\SL_d$]
\label{drift_sld}
Let $\mu$ be a probability measure on $G=\SL_d(\R)$ and $H$ the algebraic subgroup generated by $\mu$.
Assume that $\mu$ has a finite moment of order $1$, and that $H^{nc}$ is semisimple.
\begin{enumerate}[label=(\roman*)]
\item \label{dr}  $\forall  \lambda\in (0,1),\ \forall \alpha>0,\,\exists N\geq 0,\, \forall n\geq N,\ \exists A_{0}, A_{1}>0,\, \forall A\geq A_1$,  $\forall x\in f^{-1}_A(A_{0},+\infty)$, $$\mu^{*n}(g\in G\ |\ f_A(gx) \leq f_A(x)-n\lambda) \geq 1-\alpha$$
\item \label{cv} $\exists C>0$, $\forall A>0$, $\forall x\in \phi^{-1}_{A}(\R)$, $\forall g\in H$, $$\phi_A(gx)-\phi_A(x) \leq C \log\norm{g}$$
which implies the same inequality with $f_{A}$ instead of $\phi_{A}$. 
\end{enumerate}
\end{proposition}

The proof of this proposition is based on a convexity property of $\phi_{A}$, Lemma~\ref{smfaible} below, which is a reformulation of \cite[Lemma~4.2]{bq_fv}.
The first statement should be understood as follows: if two primitive subgroups $\Delta, \Delta'\leq x$ take large values under $\phi_{A}$, then one can obtain an even larger value by considering the intersection $\Delta\cap \Delta'$ or the sum $\Delta+ \Delta'$. 

\begin{lemma}[Weak submodularity property \cite{bq_fv}]
\label{smfaible}
For every $A_0> 0$, there exists $A_1> 0$, such that for all $A\geq A_1$, for every lattice $x\in \Omega$, all primitive  subgroups $\Delta, \Delta'\leq x$ satisfying $\Delta \subsetneq \Delta'$, $\Delta' \subsetneq \Delta$ and $\phi_A(\Delta),\phi_A(\Delta') >0$,
\[
\min\left(\phi_A(\Delta),\phi_A(\Delta')\right) + A_0 < \max\left(\phi_A(\Delta\cap \Delta'),\phi_A(\Delta+\Delta')\right).
\]
In particular, given $x\in\Omega_A$ such that $f_A(x)>A_0$, for each $i\in\{1,\dots,d-1\}$, there exists at most one primitive subgroup $\Delta_i\leq x$ of rank $i$ such that $$ \phi_A(\Delta_i) \geq f_A(x)- A_0$$
\end{lemma}
\begin{proof}
We check that \cite[Lemma~4.2]{bq_fv} does imply Lemma \ref{smfaible}. We first explain why our notations correspond to those of \cite{bq_fv}.
Fix a Cartan subspace $\mathfrak{a}$ in the Lie algebra of $H$ compatible with our choice of maximal compact subgroup $H^c$, and an open Weyl chamber $\ka^{++}\subseteq \ka$. Each isotypical component $V^{(j)}$ corresponds to a highest weight $\lambda^{(j)} \in\ka^*$, and by \cite[Lemma~8.18]{BQRW} one has
\begin{equation}\label{exli}
\lambda_1^{(j)} = \bracket{\lambda^{(j)},\sigma_\mu}
\end{equation}
where $\sigma_\mu\in\ka^{++}$ is the Lyapunov vector of $\mu$, given by the law of large numbers \cite[Theorem~10.9]{BQRW}.

Now  set  $\delta_{\lambda}= \bracket{\lambda,\sigma_\mu}$ for  $\lambda \in \{\lambda^{(j)}, j\in J\}$  and $\delta_{i}=i(d-i)$, $i\in \{1,\dots, d-1\}$.
Then, for every $A>0$ and every discrete subgroup $\Delta$ in $\R^d$ of rank $1\leq i<d$, the function $\phi_A$ can be expressed from the function $\varphi_{\eps_0}$ used in \cite{bq_fv} by
\[
\phi_{A}(\Delta)=\log \varphi_{\eps_{0}}(\Delta), \quad\mbox{where}\ \eps_{0}=e^{-A}. 
\]

Let us now check that \cite[Lemma~4.2]{bq_fv} does imply the weak submodularity property. 
Fix $A_{2}>0$ such that for $A\geq A_{2}$, any $\Delta$, the inequality $\phi_{A}(\Delta)>0$ implies $\|\Delta\|<1$. 
Let $A\geq A_{2}$, and consider $x\in \Omega$, as well as two primitive subgroups $\Delta \neq \Delta'\leq x$ which are not included in one another and such that $\phi_A(\Delta),\phi_A(\Delta')> 0$. Let  $(u_{1}, \dots, u_{r})$ be a basis for $\Delta\cap \Delta'$ and complete it into basis $(u_{1}, \dots, u_{r}, v_{1}, \dots, v_{s})$, $(u_{1}, \dots, u_{r}, w_{1}, \dots, w_{t})$ of $\Delta$ and $\Delta'$. The assumption  $\Delta \subsetneq \Delta'$, $\Delta' \subsetneq \Delta$ guarantees that $s,t\geq 1$. Note also that $(u_{1}, \dots, u_{r}, v_{1}, \dots, v_{s}, w_{1}, \dots, w_{t})$ is a basis for $\Delta+\Delta'$. 

We  now apply  \cite[Lemma~4.2]{bq_fv}  to $u=(u_{i})$, $v=(v_{j})$, $w=(w_{k})$. 
Note that case $iv)$ cannot occur as the conditions $A\geq A_{2}$ and $\phi_{A}(\Delta), \phi_{A}(\Delta')>0$ impose that if $r=0$ and $s+t=d$, then $\|\Delta+\Delta'\|=\|v\wedge w\|<1$ which is absurd for a sublattice of a unimodular lattice. 
Hence we are left with cases $i)$, $ii)$, $iii)$. As long as $A>C_{1}$ where $C_{1}>1$ is some constant depending only on $H$,
they all express the inequality 
$$\min\left(\phi_A(\Delta),\phi_A(\Delta')\right) + R(A) < \max\left(\phi_A(\Delta\cap \Delta'),\phi_A(\Delta+\Delta')\right)$$
with $R(A)=\frac{1}{2}(\max_{\lambda\in P^+} \delta_{\lambda})^{-1} (A-\log C_{1} )$. To conclude, for any $A_{0}>0$ choose $A_{1}>0$ such that $A> A_{2}, C_{1}$ and $R(A_{1})>A_{0}$ to get the announced result. 

For the second statement, assume for a contradiction that $\Delta'_i\leq\Delta$ is another primitive subgroup of rank $i$ satisfying $ \phi_A(\Delta'_i)\geq f_A(x)-A_0$.
The above inequality applied to $\Delta_i$ and $\Delta'_i$ implies
\[
f_{A}(x)\leq  \min\left(\phi_A(\Delta_{i}),\phi_A(\Delta'_{i})\right) + A_0 < \max\left(\phi_A(\Delta_{i}\cap \Delta'_{i}),\phi_A(\Delta_{i}+\Delta'_{i})\right)\leq f_{A}(x)\]
which is absurd. 
\end{proof}

We are now ready to prove Proposition~\ref{drift_sld}. 

\begin{proof}[Proof of Proposition~\ref{drift_sld}]
Item~\ref{cv} is an immediate consequence of our definitions, so we only prove \ref{dr}.

Let $\lambda\in (0,1)$ and $\alpha>0$. 
By Lemma~\ref{furst}, there exists $N\geq 0$ such that for all $n\geq N$, for all $v_1,\dots,v_d$ in $V$ with $0<\abs{v_i}<+\infty$, 
\begin{equation}\label{boundp}
\mu^{*n}\left(g\in H \ \left|\  \min_{1\leq i\leq k}\log\frac{\abs{g v_i}}{\abs{v_i}} \geq n\lambda\right.\right) \geq 1-\frac{\alpha}{2}.
\end{equation}
Fix such $n\geq N$, let $C>0$ be such that \ref{cv} holds, and choose a constant $A_{0}>0$ large enough so that  
\[
\mu^{*n}\left(g\in H \ \left|\ C\log \|g\|\leq \frac{1}{2}A_{0}\right.\right) \geq 1-\frac{\alpha}{2}.
\]
Let $A_1>0$ be the parameter given by Lemma~\ref{smfaible}, and $A\geq A_{1}$.
Then let $x\in \Omega_{A}$ such that $f_{A}(x)>A_{0}$ and $\mathscr{D}:=\{\Delta\leq x\ |\ \Delta \text{ primitive }, \,\phi_{A}(\Delta)\geq f_{A}(x)-A_{0}\}$. The second statement of Lemma~\ref{smfaible} guarantees that $\text{card}\mathscr{D}\leq d$. In particular, by our choice for $n$, and observing that $\frac{|g\Delta|}{|\Delta|}=\frac{|g\Delta|_{A}}{|\Delta|_{A}}$, 
\[
\mu^{*n}\left(g\in H \ |\ \forall \Delta\in \mathscr{D}, \,\phi_{A}(g\Delta) \leq \phi_{A}(\Delta) -n\lambda\right) \geq 1-\frac{\alpha}{2}.
\]
Moreover, as $|\phi_{A}(gx)-\phi_{A}(x)|\leq C\log \|g\|$, we have 
$$\mu^{*n}(g\in H \,|\, f_{A}(gx)=\max_{\Delta\in \mathscr{D}} \phi_{A}(g\Delta)) \geq 1-\frac{\alpha}{2}$$
which concludes the proof. 

\end{proof}

\subsection{General quotients $G/\Lambda$}

This paragraph concludes  Section \ref{Sec4}  by the proof of  \ref{thA'} for general quotients $G/\Lambda$. We  follow the argument of \cite[Section~6]{bq_fv}.  

\begin{proposition}
\label{driftcusp}
Let $(G,\Lambda, \Omega, \mu)$ be as in  \ref{thA'}. Then the $\mu$-walk on $\Omega$ has a controlled drift away from  infinity. 
\end{proposition}

\begin{proof}[Proof sketch]
\noindent\underline{First case:} $G$ is a semisimple $\Q$-group  and $\Lambda=G(\Z)$.\\ 
If $G\hookrightarrow\SL_N$ is an embedding defined over $\Q$, it induces a proper map
\[
G/\Lambda\hookrightarrow\SL_N(\R)/\SL_N(\Z),
\]
Hence the controlled drift away from infinity in $\SL_N(\R)/\SL_N(\Z)$ given by Proposition~\ref{drift_sld} induces by restriction the result on $G/\Lambda$.

\smallskip
\noindent\underline{Second case:} $G=\Aut\g$ with $\g$ semisimple of real rank $1$.\\
Here, the construction does not use the Benoist-Quint function. We recall the argument given in \cite[Section~6]{bq_fv}. By the results of Garland and Raghunathan \cite{gr} on reduction theory for lattices in semisimple Lie groups of real rank $1$ (see also \cite[Proof of Lemma~6.3, 2nd case]{bq_fv}), there exists a representation $V$ of $G$ and a finite family of vectors $v_1,\dots,v_r$ in $V$ such that
\begin{enumerate}
\item [(i)] each $\Lambda$-orbit $\Lambda v_i$ is discrete in $V$;
\item [(ii)] for each $i=1,\dots,r$, the orbit $G\cdot v_i$ is compact in $\PP(V)$ and contains no $H^{nc}$-invariant element;
\item [(iii)] a sequence $(g_n\Lambda)_{n\geq 0}$ goes to infinity in $G/\Lambda$ if and only if
\[
\lim_{n\to\infty}\left(\min_{1\leq i\leq r}\min_{\gamma\in\Lambda}\norm{g_n\gamma v_i}\right) =0;
\]
\item [(iv)] for every $A_0\geq 0$, there exists $\delta>0$ such that for every $g\in G$, if
\[
\min_{1\leq i\leq r}\min_{\gamma\in\Lambda}\norm{g\gamma v_i}=\norm{g\gamma_0v_{i_0}}\leq\delta,
\]
then for every $\gamma v_i\neq \gamma_0 v_{i_0}$, $\norm{g\gamma_0v_{i_0}}\leq e^{-A_0}\norm{g\gamma v_i}$.
\end{enumerate}
The function
\[
f(g\Lambda) = \max_{1\leq i\leq r}\max_{\gamma\in\Lambda}\log \frac{1}{\norm{g\gamma v_i}}
\]
is therefore proper on $G/\Lambda$.  One can easily adapt the proof of Proposition~\ref{drift_sld} to show that it has the properties of ``probable decrease'' and ``control of variation'' required for the controlled drift \ref{cd} away from infinity. 

\smallskip
\noindent\underline{Third case:} $G=\Aut\g$ where $\g$ is semisimple without compact ideal.\\
Replacing $\Lambda$ by a finite-index subgroup if necessary, we may decompose $\Omega=G/\Lambda$ into a finite product
\[
\Omega = \Omega_1\times\dots\times\Omega_r,
\]
where $\Omega_i=G_i/\Lambda_i$ is the quotient of a semisimple group $G_i$ without compact factors by an \emph{irreducible} lattice $\Lambda_i$.
By the Margulis Arithmeticity Theorem \cite[Theorem~1.16, page 299]{margulis_discretesubgroups}, either $G_i$ has real rank one, or $\Lambda_i$ is an arithmetic subgroup.
In both cases, we have already shown that the $\mu$-walk on $\Omega_{i}$ has a controlled drift away from infinity. This is equivalent to saying that the random walk  on $\Omega$ has a controlled drift away from $\cS_{i}=(\infty, S_{i})$ where $S_{i}$ is the product of the $(\Omega_{j})_{j\neq i}$ and $\infty_{i}$, the point at infinity in $\Omega_{i}$. Proposition \ref{drift-union} then implies that the $\mu$-walk on $\Omega$ drifts away from  $ \cS_{1}\vee\dots \vee  \cS_{r}$, i.e. from $(\infty, \infty)$.

\smallskip
\noindent\underline{Fourth case:} General case.\\
Let $\Ad_{\ks}:G\to\Aut\ks$ be the adjoint action of $G$ on $\ks$.
By \cite[Lemma~6.4]{bq_sadic}, the image $\Lambda_S=\Ad_{\ks}\Lambda$ is a lattice in $S$, and the induced map $p_S:\Omega\to \Aut\ks/\Lambda_S$ is proper. 
By what precedes, the walk on $\Aut\ks/\Lambda_S$  drifts away from infinity, and denoting by $f^{(S)}$ the drift function involved, we obtain a drift function for $G/\Lambda$ by setting 
\[
f = f^{(S)}\circ p_S.
\]
\end{proof}

\section{Non-accumulation on invariant subspaces}
\label{Sec5}

The goal of this section is to prove \ref{thB'}, stating that a random walk on $G/\Lambda$ does not accumulate on a proper invariant subset unless it is stuck inside it. The statement is slightly more technical than its naive counterpart Theorem \ref{thB} of the introduction. It is indeed the version we need to show that the equidistribution results of Benoist and Quint \cite{bq3} hold under a finite first moment assumption.
The proof of \ref{thB'} is based on the construction of an appropriate drift function as in Definition \ref{cd}.

\begin{named}{Theorem B$^\prime$}\emph{(Unstability of  invariant homogeneous subspaces)}
\label{thB'}
Let $G$ be a real Lie group, $\Lambda$ a lattice in $G$, and $\Omega=G/\Lambda$.  Let $\mu$ be a probability measure on $G$ and denote by $H\subseteq \emph{Aut}(\g)$  the real algebraic group generated by the support of $\Ad_{*}\mu$. We assume that $\mu$ has a finite first moment and that $H^{nc}$ is semisimple. Denote by $L$ the centralizer of $H^{nc}$ in $G$.  
Let $Y\in\cS_\Omega(\Gamma)$,  and consider compact subsets $L_{0}\subseteq L$, $K\subseteq \Omega\smallsetminus L_{0}Y$.

Then, for every $\eps>0$, there exists a neighborhood $O'$ of $L_{0}Y$ in $\Omega$  such that for all $x\in K$, 
\begin{enumerate}[label=(\roman*)]
\item for every $n\geq 0$, $(\mu^{*n}*\delta_x)(O')<\eps$;
\item for $\mu^{\otimes\N^*}$-almost every instructions $(g_i)_{i\geq 1}$,
\[
\limsup_{n\to+\infty}\frac{1}{n}\card\{k\in\{1,\dots,n\}\ |\ g_k\dots g_1x\in O'\}<\eps.
\]
\end{enumerate}
\end{named}
\bigskip

\begin{remark}

1) More precisely, the notation $L$ refers to the subgroup of $G$ defined by the Lie algebra $\kl=\{v\in \g \,|\, \forall a\in H^{nc},\,  a(v)=v\}$. Note that this subgroup may not be closed in $G$.

2) In the context of  \ref{thB'},  for any $Y\in S_{\Omega}(\Gamma)$, $x\in \Omega \smallsetminus LY$, we have 
\begin{enumerate}
\item Every weak limit $\nu$ of $(\mu^{*n}*\delta_{x})_{n\geq 0}$ satisfies $\nu(LY)=0$.
\item For $\mu^{\otimes\N^*}$-almost every trajectory $(g_n)_{n\geq 1}$,  every weak limit $\nu$ of the sequence of empirical measures $(\frac{1}{n}\sum_{k=0}^{n-1}\delta_{g_k\dots g_1x})_{n\geq 1}$ satisfies
\[
\nu(LY) = 0.
\]
\end{enumerate}
\end{remark}

To prove  \ref{thB'}, we show that the $\mu$-walk on $\Omega$ drifts away from $L$-neighborhoods of $Y$.  

\begin{proposition}
\label{drifty}
Let $(G,\Lambda, \Omega, \mu)$ be as in  \ref{thB'}. Let $\cS_{L}$ be the collection of pairs $(Y,UY)$ where $Y\in S_{\Omega}(\Gamma)$ and $U\subseteq L$ is a neighborhood of $1\in L$. 

Then the $\mu$-walk on $\Omega$ has a controlled drift away from $\cS_{L}$. 
\end{proposition}

It is easy to check that Proposition \ref{drifty} implies  \ref{thB'}.

\begin{proof}[Proof of  \ref{thB'}]
For every $l\in L_{0}$, we have $K\subseteq \Omega\smallsetminus lY$. Proposition \ref{drifty} implies that the $\mu$-walk on $\Omega$ has a controlled drift away from the collection $\cS$ of $\{(lY, UlY)\}$ where $U$ varies among the neighborhoods of $1$ in $L$. Choose an  arbitrary  $U_{l}$ such that the conclusions of Proposition \ref{cd-rec} hold for the $\mu$-walk on $\Omega$.  As $L_{0}$ is compact, there exists a finite covering $L_{0}\subseteq U_{l_{1}}l_{1}\cup \dots \cup U_{l_{m}}l_{m}$.   Let $\eps>0$, and for $i=1,\dots, m$, choose a neighborhood $O'_{i}$ of $U_{l_{i}}l_{i}Y$ as in Proposition \ref{cd-rec} for the constant  $\eps/m$. Then $O'=\cup_{i=1}^mO'_{i}$ is  a neighborhood of $ L_{0}Y$  which satisfies the recurrence properties $(i)$ and $(ii)$ announced in  \ref{thB'}. 
\end{proof}

The rest of the section is dedicated to the proof of  Proposition \ref{drifty}. It can be decomposed into two independent statements. 

\begin{lemma}  \label{ld-UY}
The $\mu$-walk on $\Omega$ has a controlled drift away from $\cS_{L}$ on compact subsets of $\Omega$. 
\end{lemma}

\begin{Lemma-drift-union}
In the general context of Section \ref{Sec3}, if a random walk has  controlled drift away from infinity and controlled drift from some $\cS\subseteq \mathscr{P}(\overline{\Omega})\times \mathscr{P}(\overline{\Omega})$ on compact subsets, then it has controlled drift away from $\cS$.  
\end{Lemma-drift-union}

\subsection{Local drift}

Let us prove Lemma \ref{ld-UY}. To construct  a local drift function  for some $UY$, we exhibit a transverse direction to $UY$ which is expanded by the walk.  The Lie algebra  $\mathfrak{g}$ of $G$ is endowed  with an arbitrary norm, and given a subspace $\mathfrak{t}\subseteq \mathfrak{g}$, we set $B_{\kt}(0,\delta)=\{v\in \kt\ |\ \|v\|<\delta \}$. 
\begin{lemma}[Transverse coordinate] 
\label{transverse}
Let $Y\in S_{\Omega}(\Gamma)$, denote by $\ks$ be the Lie algebra of $G_Y=\Stab_G Y$,  $\kl$ the Lie algebra of $L$, and fix an $H$-invariant subspace $\kt\subseteq\g$ such that
\[
\g = \kt \oplus (\kl+\ks).
\]
For any compact subset $M\subset\Omega$,  there exists a  neighborhood $U$ of $1$ in $L$ and $\delta>0$ such that for every $x\in M$, 
there is at most one $v\in B_{\kt}(0,\delta)$ for which $x\in e^vUY$.
\end{lemma}
\begin{proof}[Proof of  Lemma \ref{transverse}]

Let $\kl'\subset\kl$ be a subspace such that $\kl+\ks=\kl'\oplus\ks$, and $O$ a relatively compact open neighborhood of $M$ in $\Omega$.
There exist small neighborhoods of the origin $V_{\kt}\subset\kt$, $V_{\kl'}\subset\kl'$ and $V_{\ks}\subset\ks$ such that
\[
\begin{array}{ccc}
V_{\kt}\times V_{\kl'}\times (Y\cap O) & \to & \Omega\\
(v_{\kt},v_{\kl'},y) & \mapsto & e^{v_{\kt}}e^{v_{\kl'}} y
\end{array}
\]
induces a diffeomorphism on its image and $e^{V_{\kl'}}e^{V_{\ks}}$ contains a neighborhood $U$ of $1$ in $L$.

Moreover, reducing those neighborhoods if necessary, we may assume that
\[
e^{V_{\ks}}e^{-V_{\kl'}}e^{-V_{\kt}}M \subset O.
\]
Given $x\in M$, assume that for some $v_i\in V_{\kt}$, $u_i\in U_0$ and $y_i\in Y$, $i=1,2$,
\[
x = e^{v_1}u_1y_1 = e^{v_2}u_2y_2.
\]
Writing $u_i=u_i's_i$, with $u_i'\in e^{V_{\kl'}}$ and $s_i\in e^{V_{\ks}}$, we find
\(
e^{v_1}u_1'(s_1y_1)=e^{v_2}u_2'(s_2y_2)
\)
and since $s_iy_i=s_iu_i^{-1}e^{-v_i}x\in Y\cap O$, this implies $u_1'=u_2'$, $s_1y_1=s_2y_2$, and $v_1=v_2$.
\end{proof}

\bigskip
We now use the previous lemma to prove  local drift away from some $UY$. Given $Y\in S_{\Omega}(\Gamma)$, and  $U\subseteq L$ a  neighborhood of $1\in L$, we introduce for every  $x\in\Omega$,  the set $\kt_x=\{v\in B_{\kt}(0,1) \ |\ x\in e^vUY\}$ and define a function on $\Omega$ by

\[
f_{UY}(x) = \left\{
\begin{array}{ll}
 \max_{v\in\kt_x}\log \frac{1}{\norm{v}} & \mbox{if}\ \kt_x\neq\varnothing\\
0 & \mbox{otherwise}.
\end{array}
\right.
\]

The proof of Lemma \ref{ld-UY} will follow from the observation that for any $Y\in S_{\Omega}(\Gamma)$ and any compact set $Q\subseteq \Omega$, we can choose $U$ small enough so that the function $f_{UY}$ statistically decreases when the walk passes through $Q$.

\begin{proof}[Proof of Lemma \ref{ld-UY}]
Let $\mathfrak{b}\subseteq \g$ be the unique $H$-invariant subspace such that  $\mathfrak{g}=\mathfrak{b}\oplus \mathfrak{l}$.   Since $\mathfrak{b}$ intersects $\kl$ trivially, the law of large numbers \cite[Theorem 4.28]{BQRW} and Furstenberg's theorem on the positivity of the first Lyapunov exponent \cite[Theorem 4.32]{BQRW} show that there exists $\lambda>0$ such that for every $\alpha>0$, there exists $N\geq 0$ such that  for every $n\geq N$, and every $v\in \mathfrak{b}\setminus\{0\}$,
\begin{equation}\label{lgn}
\mu^{*n}\left(g \in G\ \left|\ \log\frac{\norm{g v}}{\norm{v}} \geq n\lambda \right.\right) \geq 1-\frac{\alpha}{3}.
\end{equation}

Fix $n\geq N$. Let $Y\in S_{\Omega}(\Gamma)$ and consider $K, Q\subseteq \Omega$ compact sets with $K\cap Y=\varnothing$.  We can assume $\kl+ \ks\neq \g$ otherwise $UY$ is open in $\Omega$ for any open set $U\subseteq L$, and the drift statement follows trivially by letting $f$ be infinite on $UY$ and null elsewhere. Consider a  compact set $M\subseteq \Omega$ such that 
\begin{equation}\label{M}
\mu^{*n}(g\in G\ |\ gQ\subseteq M) \geq 1-\frac{\alpha}{3}
\end{equation}
and choose $\kt\subseteq \mathfrak{b}$, $U\subseteq L$, $\delta>0$  as in Lemma \ref{transverse}, with $U$ open relatively compact in $L$, and so that $K\cap \overline{U}Y=\varnothing$. Set $Y'=UY$ and $f=f_{UY}$. We show that $f$ satisfies the properties required for controlled drift. The fact that $f$ is bounded on $K$ comes from the observation that $\overline{U}Y$ is closed which, combined with the condition  $K\cap \overline{U}Y=\varnothing$, imposes that the distance between $K$ and $UY$ is bounded below by some positive constant.
The condition on upper-level sets is straightforward. 

Let us check the condition \ref{pdec} of \ref{cd}. Set $T>0$ so large that
\begin{equation}\label{n}
\mu^{*n}(g\in G\ |\ \norm{\Ad g}< e^T\delta) \geq 1-\frac{\alpha}{3}.
\end{equation}
 Let  $x\in Q$ such that  $T\leq f(x)<+\infty$, i.e. $x\in e^vUY$ with $0<\norm{v}\leq e^{-T}$.
Let $g\in G$ be any element satisfying  conditions \eqref{lgn}, \eqref{M}, and \eqref{n}.  By \eqref{n} one has $gx \in e^{\Ad(g)v}UY$ with $\norm{\Ad(g)v}\leq\norm{\text{Ad}g}\norm{v}<\delta$. As $gx\in M$ by  \eqref{M}, our choice for $(U,\delta)$ yields that $\kt_{gx}=\{\Ad(g)v\}$, hence $f(gx)=\log \frac{1}{\|\Ad(g)v\|}$. By  \eqref{lgn}, 
\[
f(gx) \leq f(x) - n\lambda.
\]

To check \ref{cvar},  observe that for $g\in \Gamma$, one has $\norm{\Ad g^{-1}}^{-1} \,\|v\|\leq \norm{\Ad g v}$ and $\norm{\Ad g^{-1}}\leq  \norm{\Ad g}^{\kappa}$ for some constant $\kappa>0$ depending only on the normed vector space $(\mathfrak{g}, \|.\|)$.
So $C=\kappa$ leads to the bound announced in the lemma. 

\end{proof}

\subsection{From local drift to global drift}

\begin{proof}[Proof of Lemma  \ref{drift-union} b)]

\label{proof-du-ii}
In this proof, the notations of \ref{cd} referring to the drift away from infinity will be indexed by $0$, and those referring to the drift away from $\cS$ on compact subsets will be indexed by $1$. Now let $(\Phi_{i},\lambda_{i}, C_{i})_{i=0,1}$ be given by \ref{cd}. Set $c>1$ such  that $c>\lambda_{0,\infty}$ and $\frac{c\lambda_{0}}{2}> 2C_{1}\lambda_{1, \infty}$ where $\lambda_{i, \infty}\geq 0$ is the first Lyapunov exponent of $\Phi_{i \ast} \mu$. Fix  $\alpha>0$ and denote by $N_{0},N_{1}\geq0$ the associated constants in \ref{cd}.  Fix $n\in \N$ large enough so that $n\geq \max(N_{0}, N_{1})$ and 
\[
\mu^{\ast n} (g\in G \,|\, \log\|\Phi_{0}(g)\|\leq c n)\geq 1-\alpha
\]
\[
\mu^{\ast n} (g\in G \,|\, \log\|\Phi_{1}(g)\|\leq \frac{c\lambda_{0}}{2C_{1}}n)\geq 1-\alpha
\]
Let $Y\in S_{\Omega}(\Gamma)$, let $K\subseteq \Omega \smallsetminus Y$ be a compact set. The controlled drift away from infinity associates to $K$ a proper drift function $f_{0}:\Omega\rightarrow[0,+\infty]$ and a threshold $T_{0}>0$ as in  \ref{cd}. Let $T'>0$ be a large parameter  which will be specified later.  Choose a compact subset $Q_{1}\subseteq \Omega$ containing $K\cup \{f_{0}\leq T'\}$, and consider the objects  $Y'\in S_{Y}$, $f_{1}:\Omega\rightarrow[0,+\infty]$, $T_{1}>0$ associated to $(K,Q_{1})$ as in \ref{cd} for the drift away from $\cS$ on compact sets. 

We are now all set up to define our drift function:  for $x\in \Omega$, set
\[
f(x) = c_x f_0(x) + f_{1}(x) 
\]
where
\[
c_x = \left\{
\begin{array}{ll}
c & \mbox{if}\ f_0(x)\geq T'\\
c\cdot\frac{f_0(x)-T_{0}}{T'-T_{0}} & \mbox{if}\ T_{0} \leq f_0(x)\leq T'\\
0 & \mbox{if}\ f_0(x)\leq T_{0}.
\end{array}
\right.
\]

Let us check that $f$ satisfies the properties enumerated in Definition \ref{cd}.  The only non-trivial conditions to check are those enumerated as \ref{pdec} and \ref{cvar}. 

\bigskip
\noindent\textbf{Probable decrease}\\
We prove that if the parameter $T'$ above is chosen sufficiently large, then $f$ satisfies the  property \ref{pdec} in \ref{cd}, with $T=cT'+T_{1}$ and $\lambda= \frac{1}{2}\min(\lambda_{0}, \lambda_{1})$. Let $x\in \Omega$  such that $f(x)\in [T,+\infty)$. 
 We study separately  different cases, according to the values of $f_0(x), f_{1}(x)$. Observe  that the definition of $T$ imposes that $f_0(x)>T'$ or $f_{1}(x)>T_{1}$.

\begin{enumerate}[label=(\alph*)]
\item \underline{$f_0(x)>T'$}\\

In this case, the conditions
\(
\left\{
\begin{array}{l}
c_{gx}f_0(gx)\leq cf_0(x)-nc\lambda_{0}\\
C_{1}\log\norm{\Phi_{1}(g)}\leq \frac{nc\lambda_{0}}{2}
\end{array}
\right.
\)
are simultaneously satisfied with $\mu^{\ast n}$-probability at least $1-2\alpha$ and ensure that 
$$f(gx)\leq f(x)-\frac{nc\lambda_{0}}{2}$$

\item \underline{$T_{0}<f_0(x)\leq T'$ and $f_{1}(x)>T_{1}$}\\
Then, with $\mu^{\ast n}$-probability at least $1-2\alpha$, one has both
\(
\left\{
\begin{array}{l}
f_0(gx)\leq f_0(x)-n\lambda_{0}\\
f_{1}(gx) \leq f_{1}(x)-n\lambda_{1}.
\end{array}
\right.
\)
This certainly implies $c_{gx}\leq c_x$ and therefore
\[
f(gx)\leq f(x)-n\lambda_{1}.
\]

\item \underline{$f_0(x)\leq T_{0}$ and $f_{1}(x)>T_{1}$}\\
With $\mu^{\ast n}$-probability at least $1-2\alpha$, one has
\(
\left\{
\begin{array}{l}
f_0(gx)\leq T_{0}+ C_{0}c n \\
f_{1}(gx) \leq f_{1}(x)-n\lambda_{1}.
\end{array}
\right.
\)

This implies
\[
c_{gx}f_0(gx) \leq c\cdot\frac{C_{0}cn}{T'-T_{0}}f_0(gx) \leq c\cdot\frac{C_{0}cn}{T'-T_{0}}(T_{0}+C_{0}cn) \leq \frac{n\lambda_{1}}{2}
\]
up to choosing $T'$ large enough in the definition of $f$,  
and therefore
\[
f(gx) \leq  f(x)-\frac{n\lambda_{1}}{2}.
\]

\end{enumerate}
Putting together (a), (b), (c), we obtain for every $ x\in f^{-1}[T,+\infty)$,
\[
\mu^{*n}(g\in G\ |\ f(gx)\leq f(x)-n\lambda) \geq 1-2\alpha
\]
 where $\lambda= \frac{1}{2}\min(\lambda_{0}, \lambda_{1})>0$. 

\bigskip
\noindent\textbf{Control of variations}\\
  We show that if $T'$ is chosen large enough in the definition of $f$, then we have the bound \ref{cvar} of \ref{cd}.  with $C=3cC_{0}+C_{1}$ and $\Phi=\Phi_{0}\oplus \Phi_{1}$.  Let $x\in\Omega$, $g\in \Gamma$. We can write \[
f(gx)-f(x) =  \left(c_{gx}f_0(gx)-c_xf_0(x)\right)+  \left(f_{1}(gx)-f_{1}(x)\right).
\]
Definition \ref{cd} guarantees that
$$f_{1}(gx)-f_{1}(x)\leq C_{1}\log \|\Phi_{1}(g)\|$$
so it remains to bound the first  term. We can assume that   $f_0(x)\leq f_0(gx)$.
Then,
 \begin{align*}
c_{gx}f_0(gx)-c_xf_0(x) &= c_{gx}\left(f_0(gx)-f_0(x)\right)+ f_0(x)(c_{gx}-c_x) \\
&\leq  c \,C_{0}\log \|\Phi_{0}(g)\|+ f_{0}(x)(c_{gx}-c_x)
 \end{align*}
 If  $f_{0}(x)>T'$, then $c_{x}=c_{gx}=c$ and we are done. Otherwise,  observing that
 \[
c_{gx}-c_x \leq  c\cdot\frac{f_0(gx)-f_0(x)}{T'-T_{0}} \leq c\cdot\frac{\,C_{0}\log\norm{\Phi_{0}(g)}}{T'-T_{0}}
\]
  we obtain 
  \begin{align*}
 f_0(x)(c_{gx}-c_x)&\leq \frac{T'}{T'-T_{0}}c\,C_{0}\log\norm{\Phi_{0}(g)}\\
&\leq 2c\,C_{0}\log\norm{\Phi_{0}(g)}
 \end{align*}
provided $T'$ is chosen large enough in the definition of $f$. 
In any case, we conclude that for $x\in f^{-1}[0,+\infty)$, $g\in \Gamma$, 
 $$
f(gx)-f(x) \leq  3cC_{0} \log \|\Phi_{0}(g)\|+ C_{1} \log \|\Phi_{1}(g)\|.$$  

\end{proof}

\section{Equidistribution}
\label{Sec6}

The goal of this section is to establish Theorem \ref{thC}. We shall need the following lemma, obtained by Benoist and Quint \cite[Proposition~2.1]{bq3} in the greater generality of $S$-adic Lie groups, but with the additional assumption that $\Gamma$ is compactly generated.
The proof is almost the same as in \cite{bq3}, but we include it to explain in detail how to modify their argument to avoid this assumption.

\begin{lemma}
\label{countable}
Let $G$ be a real Lie group, $\Lambda$ a lattice in $G$, $\Gamma$ a subgroup of $G$ and denote by $L$ be the centralizer of $\Gamma$ in $G$.
Assume that $\overline{\Ad\Gamma}^Z$ is Zariski connected semisimple and without compact factors.
Then the set $\cSS$ is a countable union of $L$-orbits.
\end{lemma}

\begin{proof}
Fix a dense countable subset $D\subset\Gamma$.
Given $Y\in\cSS$, there exists a finite family of elements $g_1,\dots,g_s$ in $D$ generating a group $\Gamma_f$ such that $\overline{\Ad\Gamma_f}^Z=\overline{\Ad\Gamma}^Z$ and that moreover $\Gamma_f$ acts transitively on the (finite) set of connected components of $Y$.
Let $g\in G$ be such that $g\Lambda\in Y$, and write
\[
H = g^{-1}\Gamma_f G_Y^\circ g.
\]
Since $H\Lambda=g^{-1}Y$ has finite volume, $H\cap\Lambda$ is a lattice in $H$, and $H^\circ\cap\Lambda$ is a lattice in $H^\circ$.
Because $H$ and $H^\circ$ are compactly generated, the groups $\Sigma=H\cap\Lambda$ and $\Delta=H^\circ\cap\Lambda$ are both finitely generated, and since $\Lambda$ contains only countably many finitely generated subgroups, we may assume that they are fixed.
(The finitely generated group $\Gamma_f$ was introduced precisely to ensure that $H$ be compactly generated, so we can use the argument of \cite{bq3}.)
Then, $H$ belongs to the set $\cT(G,\Delta,\Sigma)$ of closed subgroups such that
\begin{enumerate}[label=(\roman*)]
\item $\Sigma$ is a lattice in $H$;
\item $\Delta=\Sigma\cap H^\circ$ is a lattice in $H^\circ$;
\item there exists a subgroup $\Gamma'$ in $G$ such that $\overline{\Ad\Gamma'}^Z$ is connected semisimple without compact factors and acts ergodically on $H/\Sigma$.
\end{enumerate}
By \cite[Lemma~2.5]{bq3}, this set is countable, so we may assume that $H$ is fixed.
Let $L_f$ denote the centralizer of $\Gamma_f$ in $G$, and note that $L_f^\circ=L^\circ$ as their Lie algebras are given by the sets of invariant vectors of $\overline{\Ad{\Gamma_{f}}}^Z=\overline{\Ad{\Gamma}}^Z$.
By \cite[Lemma~2.2]{bq3}, the set of fixed points of $\Gamma_f$ in $G/H$ is a countable union of $L_f^\circ$-orbits, and therefore also a countable union of $L$-orbits.
Since $Y = gH\Lambda$ and $gH$ is a fixed point of $\Gamma_f$ in $G/H$, the lemma follows.
\end{proof}

\begin{proof}[Proof of Theorem~\ref{thC}] With Theorem \ref{thA}, \ref{thB'}, Lemma \ref{countable}, and \cite[Theorem~1.3]{el_rw} at hand, the theorem follows from the argument given in \cite[\S2.3 and \S4.1]{bq3}; we include the proof for completeness.
We show the following assertion, which implies items \ref{oc}, \ref{ce} and  \ref{tr}: for every $x\in \Omega$, there exists a $\Gamma$-invariant ergodic finite-volume  homogeneous closed subset $Y$ containing $x$ such that $\left(\frac{1}{n}\sum_{k=0}^{n-1}\delta_{g_k\dots g_1x}\right)_{n\geq 1}$ converges to $\nu_Y$.

Let $(g_n)_{n\geq 1}$ be a $\mu^{\otimes\N^*}$-typical sequence of instructions, and $\nu$ any weak limit of the sequence $\left(\frac{1}{n}\sum_{k=0}^{n-1}\delta_{g_k\dots g_1x}\right)_{n\geq 1}$.
By Theorem~\ref{thA}, $\nu$ is a probability measure.
Moreover, by Breiman's law of large numbers \cite{Brei} (see also \cite[Lemma~3.2]{bq3}) the measure $\nu$ is $\mu$-stationary.
Consider a disintegration
\[
\nu = \int \nu_\alpha \dd\PP(\alpha)
\]
where each $\nu_\alpha$ is an ergodic $\mu$-stationary measure.
By the classification of stationary measures \cite[Theorem~1.3]{el_rw}, each $\nu_\alpha$ is equal to a homogeneous measure $\nu_{Y_\alpha}$ for some invariant ergodic finite-volume homogeneous closed subset $Y_\alpha$.
Using Proposition~\ref{countable}, we may rewrite this integral as a countable sum
\begin{equation}\label{dis}
\nu=\sum_{i\in\N}\nu_i,
\end{equation}
where each $\nu_i$ is a $\mu$-stationary measure supported on $LY_i$, for some $Y_i\in\cS_\Omega(\Gamma)$.
If $x\not\in LY$, \ref{thB'} implies that $\nu(LY)=0$, so there must exist $Y\in\cS_\Omega(\Gamma)$ such that $x\in Y$.
Choose such $Y$ of minimal dimension.
Replacing $G$ by $G_Y=\{g\in G\ |\ gY=Y\}$, the proof boils down to the case where $\Omega\in S_{\Omega}(\Gamma)$, $x$ does not belong to any proper $Z\in S_{\Omega}(\Gamma)$,  and we need to show that $\nu=\nu_{\Omega}$. This comes directly from a second application of the decomposition \eqref{dis} combined with  \ref{thB'}.
\end{proof}

\section{Conclusion}
\label{Sec7}

To conclude, we mention a few problems that were not discussed here, but might lead to interesting continuations of our study.

\begin{description}
\item [Compact factors.]
When the group $H=\overline{\Ad\Gamma}^Z$ is semisimple but is allowed to have compact factors, it is still possible to describe all ergodic $\mu$-stationary measures on $\Omega$.
Some might not be homogeneous, but they can nevertheless be written as an integral average of homogeneous measures \cite{el_rw}. It should be possible to obtain a generalization of Theorem~\ref{thC} to this setting.
\item [Drift function on $G$.]
On a general quotient $G/\Lambda$, the Benoist-Quint drift function is essentially constructed by embedding $G/\Lambda$ into some space of lattices $\SL_d(\R)/\SL_d(\Z)$.
It would be desirable to have an intrinsic construction of such a function, at least in the case where $G$ is a linear algebraic group defined over $\Q$ and $\Lambda=G(\Z)$ an arithmetic subgroup.
For example, this would lead to simple formulas for the optimal return time to a compact set, similar to those available for $\SL_d(\R)/\SL_d(\Z)$.
\item [Unipotency assumption.]
It was originally suggested in \cite{EskMar04} that Theorem~\ref{thA} could hold under the assumption that $H=\overline{\Ad\Gamma}^Z$ is generated by unipotents.
As observed by Emmanuel Breuillard \cite[Proposition~10.4]{breuillard}, this is not true for conclusion $(i)$ in general.
But this might be the case for $(ii)$, and $(i)$ would then hold for the Cesàro averages $\frac{1}{n}\sum_{k=0}^{n-1}\mu^{*k}*\delta_x$.
In a slightly different direction, the theorem could be valid as stated if $H$ is assumed to be perfect, i.e. if its Lie algebra $\h$ satisfies $\h=[\h,\h]$.
\item [No moment assumption.]
As suggested at the end of \cite{benoist_rec}, one may hope that the results on homogeneous random walks obtained here, and in particular Theorem~\ref{thA}, are still valid without any moment assumption.
\end{description}

\bibliographystyle{abbrv}

\bibliography{bibliography}
\end{document}